\documentclass[11pt, reqno]{amsart}
\usepackage{amssymb}
\usepackage{hyperref}
\usepackage{tikz}
\usepackage{tikz-cd}
\usepackage[margin=1in]{geometry}
\usepackage{array}
\usepackage{enumitem}
\usepackage{mathrsfs}

\usepackage[backend=biber,style=alphabetic,maxnames=99,minnames=3]{biblatex}
\usepackage{listings}
\usepackage{xcolor}

\definecolor{codegreen}{rgb}{0,0.6,0}
\definecolor{codegray}{rgb}{0.5,0.5,0.5}
\definecolor{codepurple}{rgb}{0.58,0,0.82}
\definecolor{backcolour}{rgb}{0.95,0.95,0.92}

\lstdefinestyle{mystyle}{
    backgroundcolor=\color{backcolour},   
    commentstyle=\color{codegreen},
    keywordstyle=\color{magenta},
    numberstyle=\tiny\color{codegray},
    stringstyle=\color{codepurple},
    basicstyle=\ttfamily\lineskiplimit0pt\footnotesize,
    breakatwhitespace=false,         
    breaklines=true,                 
    captionpos=b,                    
    keepspaces=true,                 
    numbers=left,                    
    numbersep=5pt,                  
    showspaces=false,                
    showstringspaces=false,
    showtabs=false,                  
    tabsize=1,
    aboveskip=0pt,
    belowskip=0pt,
    lineskip=0pt,
}
\newtheorem{theorem}{Theorem}[section]
\newtheorem*{theorem*}{Theorem}
\newtheorem{lemma}[theorem]{Lemma}

\newtheorem{corollary}[theorem]{Corollary}
\newtheorem{proposition}[theorem]{Proposition}

\theoremstyle{definition}

\newtheorem*{question}{Question}
\newtheorem{examplex}[theorem]{Example}

\newenvironment{example}
  {\pushQED{\qed}\examplex}
  {\popQED\endexamplex}

\theoremstyle{remark}

\newcommand\CC{\mathbb{C}}

\newcommand\PP{\mathbb{P}}
\newcommand\QQ{\mathbb{Q}}

\newcommand\ZZ{\mathbb{Z}}

\newcommand{\B}{\mathcal{B}}
\newcommand{\C}{\mathcal{C}}
\newcommand{\D}{\mathcal{D}}
\newcommand\E{\mathcal{E}}

\renewcommand\O{\mathcal{O}}

\newcommand\sC{\mathscr{C}}
\newcommand\aut{\mathrm{Aut}}

\newcommand\inv{^{-1}}
\newcommand\lcm{\operatorname{lcm}}

\newcommand\deck{\mathrm{Deck}}
\newcommand{\gen}{\mathrm{genus}}

\begin{document}

\title{Rational Orbits Under Correspondences}

\author{Trevor Hyde}
\address{Dept. of Mathematics and Statistics\\
Vassar College\\
Poughkeepsie, NY 12604\\
}
\email{thyde@vassar.edu}

\maketitle

\section{Introduction}

In arithmetic dynamics we typically study number theoretic properties of the polynomial or rational function iteration.
In this paper we study the arithmetic of the dynamics of algebraic functions.
Suppose $F : \PP^1 \rightsquigarrow \PP^1$ is an algebraic function defined over $\QQ$.
That is, suppose there is a polynomial $f(x,y) \in \QQ[x,y]$ such that $f(x,F(x)) = 0$.
If $\deg_y f \geq 2$, then $F$ is technically not a function of $x$, but we may view it as a multi-valued function.
For example, consider $F(x) = \sqrt{x^3 - 1}$ which is defined by $y^2 - x^3 + 1 \in \QQ[x,y]$.
The presence of the square root in $F$ suggests that even if $p \in \PP^1(\QQ)$, we typically cannot expect $F(p) \in \PP^1(\QQ)$.
And yet, miracles happen: for example, $F(1) = 0$.
Our goal is to address the following question.

\begin{question}
    How many rational points can there be in the orbit of $p \in \PP^1(\QQ)$ under an algebraic function $F$?
\end{question}

To properly formalize the dynamics of algebraic functions we work with \emph{correspondences}.
A correspondence $F : \PP^1 \rightsquigarrow \PP^1$ defined over a field $K$ is a (possibly reducible) $K$-algebraic curve $\C_F \subseteq \PP^1 \times \PP^1$ for which the projections onto each component are finite and surjective.
We may view $F(p)$ as the set of all points $q$ such that $(p,q) \in \C_F$.
There is a natural way to compose correspondences by taking fiber products; see Section \ref{sec: correspondences} for definitions and constructions.
Let $F^n$ denote the $n$th iterate of the correspondence $F$.
The degree of a correspondence is the pair $(d,e)$ of degrees of the two projection maps.
Here is a summarized version of our main result.

\begin{theorem}
\label{thm: main intro}
    Let $K$ be a finitely generated field of characteristic zero.
    Let $F : \PP^1 \rightsquigarrow \PP^1$ be a correspondence defined over $K$ with degree $(d,e)$ such that $d, e \geq 2$ and $\C_{F^{12}}$ is irreducible.
    If there exists an $n \geq 12$ for which there are infinitely many $p \in \PP^1(K)$ such that $F^n(p)$ contains a $K$-rational point, then $F$ belongs to an explicit list of exceptional correspondences.
\end{theorem}

This is proved in two parts as Theorem \ref{thm: main} and Theorem \ref{thm: exceptional classification}; the latter result contains the detailed list of exceptional cases.
To give a sense of this classification, suppose that $F$ is a correspondence satisfying the hypotheses of Theorem \ref{thm: main intro}.
We show that the normalization of $\C_F$ is isomorphic to $\PP^1$ and let $f, g : \PP^1 \to \PP^1$ be the composition of the normalization with the coordinate projections.
Then after a change of coordinates, and with essentially finitely many exceptions, we show that $f$ must be a power map, a Chebyshev polynomial, or a (generalized) Latt\'es map.
The rational function $g$ must fit into a semiconjugation diagram with $f$, which highly restricts $g$.

For example, Let $d \geq 2$, let $e$ be coprime to $d$, let $h(x) \in K(x)$ be a non-constant rational function which is not a $d$th power in $\overline{K}(x)$, and suppose that $m \geq 3$.
Consider the correspondence $F$ defined by the curve $\C_F \cong \PP^1$ parametrized by $(f,g) := (x^d, x^eh(x)^{d^{m-1}})$.
This pair of maps makes the following diagram commute.
\[
    \begin{tikzcd}
    \mathbb{P}^1 \arrow[d, "x^d"'] & & \mathbb{P}^1 \arrow[ll, "x^eh(x^d)^{d^{m-2}}"'] \arrow[d, "x^d"] \\
    \mathbb{P}^1                   & & \mathbb{P}^1 \arrow[ll, "x^eh(x)^{d^{m-1}}"]                    
    \end{tikzcd}
\]
In Theorem \ref{thm: exceptional classification} we show that $\C_{F^n}$ is irreducible for all $n \geq 0$, and that the normalization of $\C_{F^n}$ has genus 0 for $n \leq m$, and has genus $> 1$ for $n > m$.
Thus we may find infinitely many $K$-rational points whose first $m$ iterates are all in $K$, but Faltings's theorem implies only finitely many will have $K$-rational iterates beyond the $m$th iterate.

This example highlights how Faltings's theorem allows us to translate infinitely many $K$-rational iterates into the purely geometric hypothesis that $\C_{F^n}$ has an irreducible component of genus at most one.
The main source of tension that we exploit stems from how restrictive it is for a pair of finite maps to $\PP^1$ to have an irreducible fiber product with genus at most one.
This tension manifests as the following result.

\begin{theorem}
\label{thm: intro thesis}
    Let $f : \C \to \PP^1$ be a finite degree $d$ map from a smooth, complex, projective curve $\C$ and let $g : \PP^1 \to \PP^1$ be a degree $e$ rational function.
    If $e > 170d - 84$ and the fiber product of $f$ and $g$ is irreducible with normalization of genus at most one, then the Galois closure of $f$ has genus at most one.
\end{theorem}

Theorem \ref{thm: intro thesis} is from Chapter 5 of the author's Ph.D. thesis \cite{thesis}, which was joint work with Michael Zieve.
Theorem \ref{thm: intro thesis} is simplified for ease of notation; see Theorem \ref{thm: bdd genus} and Corollary \ref{corollary: low genus criteria} for the full statement and proof.

Maps $f : \C \to \PP^1$ with a Galois closure of genus at most one are, in a sense, well-known and classified.
This classification stems from the classification of finite groups of M\"obius transformations, which are the Galois groups in the genus zero Galois closure case, and the theory of elliptic curves.
In Section \ref{sec: low-genus} we carry out this classification in order to tabulate all the possible ramification types of such maps (see Table \ref{table: gen 0 census} and Table \ref{table: genus 1 gal closure}) which we require to prove Theorem \ref{thm: exceptional classification}.
Given how often these maps arise as interesting exceptional cases, these tables may be of broader interest.

Another tool which we only use in a minor way but deserves to be better known is the following theorem due to Michael Fried.

\begin{theorem}[{\cite[Prop. 2]{Fried}}]
\label{thm: frieds thm intro}
    Let $f_i : \C_i \to \D$ for $i = 1, 2$ be finite maps between curves.
    There are decompositions $f_i = g_ih_i$ such that $g_1$ and $g_2$ have the same Galois closure, and the fiber product of $g_1, g_2$ has the same number of irreducible components as the fiber product of $f_1, f_2$.
\end{theorem}

Fried's statement of this result is quite different but is essentially equivalent to Theorem \ref{thm: frieds thm intro}.
We provide a proof as Theorem \ref{thm: frieds thm} via a translation to a result about finite $G$-sets.
Theorem \ref{thm: frieds thm intro} gives us a starting point for handling reducible fiber products: If the fiber product of $f_1$ and $f_2$ is reducible, then the maps $f_1$ and $f_2$ have nontrivial left factors with the same Galois closure.

\subsection{Related work}

The study of dynamics of correspondences has a long history.
For example, the arithmetic-geometric mean introduced by Lagrange and analyzed deeply by Gauss may be viewed as the dynamics of the correspondence $\C_F : 4xy^2 = (x + 1)^2$.
Within complex dynamics there is a considerable body of work on correspondences; see for example \cite{bellaiche2023self,BogomolovTschinkel2002unramified,buchacher2025finite,bullett1992critically,bullett1994gallery,delaparra,gotou2023dynamical,lomonaco2025algebraic,} and the references therein.
The study of correspondences from the perspective of arithmetic dynamics is more recent.
See Patrick Ingram's work on constructing canonical heights for correspondences \cite{ingram2019canonical}, studying analogs of post-critically finite correspondences \cite{ingram2017critical}, and on arboreal representations of correspondences \cite{ingram2018p}.

Ideally we would like to have a version of Theorem \ref{thm: main intro} which did not require the hypothesis that $\C_{F^{12}}$ is irreducible.
However, this requires an analysis of cases where $\C_{F^n}$ may be reducible with components of genus at most one.
When the projections from $\C_F$ are polynomials, the recent resolution of the Davenport-Lewis-Schinzel conjecture by Behajaina, König, and Neftin \cite{DLS} should give one a foothold in this case, but we leave that for the future.
See also the work of An and Diep \cite{an2013genus}. 

\subsection{Acknowledgments}

The inspiration for this project and my introduction to the dynamics of correspondences came from the \emph{Dynamics of Multiple Maps} workshop at the American Institute of Mathematics in November, 2025.

\section{Foundations and fiber products}

In this paper, $K$ will always be a field of characteristic zero.
With the exception of Theorem \ref{thm: main}, we assume $K$ is algebraically closed and the Lefschetz principle allows us to assume without loss of generality that $K = \CC$.

For us a \textbf{curve} $\C$ is a smooth, projective scheme defined over $K$ of pure dimension one.
Every such curve is a finite disjoint union of \textbf{irreducible curves}.
Let $\sC$ denote the category of all curves with finite maps; we will only be working with finite maps and thus simply refer to them as maps.
The category $\sC$ is dual equivalent, under the standard algebro-geometric dictionary, to the category of finite products of finitely generated extensions of $K$ with transcendence degree one and $K$-algebra homomorphisms.

\subsection{Fiber products}
Let $f_1 : \C_1 \to \D$ and $g_1 : \C_2 \to \D$ be maps between curves.
The fiber product of $f_1$ and $g_1$ is the potentially singular curve
\[
    \B := \{(p,q) \in \C_1 \times \C_2 : f_1(p) = g_1(q)\}.
\]
Let $\widetilde \B \to \B$ be the normalization of $\B$, and let $g_2 : \widetilde \B \to \C_1$ and $f_2 : \widetilde \B \to \C_2$ be the composition of the normalization with the natural projections so that $f_1 g_2 = g_1 f_2$.
The curve $\widetilde\B$ together with the maps $f_2, g_2$ is the categorical fiber product of $f_1$ and $g_1$ in $\sC$.
When we refer to the fiber product of maps between curves, we mean this normalized fiber product in $\sC$.
Abhyankar's lemma characterizes the points on a fiber product and their ramification.
If $f : \C \to \D$ is a finite map and $p \in \C$, let $e_f(p)$ denote the ramification index of $f$ at $p$.

\begin{lemma}[{Abhyankar's lemma}]
\label{lemma: abhyankar}
    Let $f_1 : \C_1 \to \D$ and $g_1 : \C_2 \to \D$ be maps between curves and let $\B$ denote the fiber product and let $h : \B \to \D$ be the map $h := f_1g_2 = g_1f_2$.
    Suppose that $(p,q) \in \C_1 \times \C_2$ is a point such that $f_1(p) = g_1(q)$.
    There are $\gcd(e_{f_1}(p), e_{g_1}(q))$ points $r \in \B$ such that $f_2(r) = q$ and $g_2(r) = p$ and
    \(
        e_h(r) = \lcm(e_{f_1}(p), e_{g_1}(q)).
    \)
\end{lemma}

\begin{proof}
    See, for example, Stichtenoth \cite[Thm. 3.9.1]{Stich}.
\end{proof}

Suppose $f_1 : \C_1 \to \D$ and $g_1 : \C_2 \to \D$ are a pair of maps between curves.
Even if all the curves are irreducible, the fiber product may be reducible.
The next lemma provides one simple sufficient condition for the fiber product to be irreducible.

\begin{lemma}
\label{lemma: coprime degree implies irreducible}
    Let $f_1 : \C \to \D$ and $g_1 : \C' \to \D$ be a pair of finite maps between irreducible curves with degrees $d, e$ respectively.
    If $d$ and $e$ are coprime, then the fiber product of $f_1$ and $g_1$ is irreducible.
\end{lemma}

\begin{proof}
    Let $\B$ be an irreducible component of the fiber product of $f_1$ and $g_1$, and suppose that $f_2, g_2$ are maps from $\B$ such that $h := f_1g_2 = g_1f_2$.
    Then $\deg h$ is divisible by both $d$ and $e$, hence by $de$ since we assumed $d$ and $e$ are coprime.
    On the other hand, let $h' : \B' \to \D$ be the map from the full fiber product of $f_1$ and $g_1$ to $\D$.
    By construction, $\deg h' \leq de$.
    But $h'$ factors through $h$, implying that $\deg h' = de$.
    Thus $h = h'$ and $\B = \B'$, which is to say that $\B'$ is irreducible.
\end{proof}

\subsection{Galois maps}

Let $f : \C \to \D$ be a map and define $\deck(f)$ to be the group of all automorphisms $g : \C \to \C$ such that $f g = f$.
Let $G \subseteq \deck(f)$ be a finite subgroup and let $\C/G$ be the normalization of the quotient of $\C$ by $G$.
We say $f$ is \textbf{Galois} with \textbf{Galois group} $G$ if there exists an isomorphism $g : \C/G \to \D$ such that $f = gh$ where $h : \C \to \C/G$ is the natural quotient map.
If $\C$ and $\D$ are irreducible, then $G = \deck(f)$, but otherwise $G$ may be a proper subgroup.
The Galois closure of $f : \C \to \D$ is the final Galois map $\bar f : \overline \C \to \D$ which factors through $f$; the existence and uniqueness up to isomorphism follows from standard Galois theory of field extensions via duality.

If $\tilde f : \B \to \D$ is a Galois map which factors through $f$ and $G$ is the Galois group of $\tilde f$, then each $g \in G$ induces an isomorphism $g : \C^g \to \C$.
Then $f^g := g\inv f g : \C^g \to \D$ is the conjugate of $f$ by $g$.
The Galois closure of $f$ may be constructed as the component of the fiber product of all the conjugates of $f$ through which $\tilde f$ factors.
The following lemma is a direct consequence of Abhyankar's lemma.

\begin{lemma}
\label{lemma: galois ramification}
    Let $f: \C \rightarrow \D$ be a map between curves and let $\bar f: \overline \C \rightarrow \D$ be the Galois closure of $f$. 
    Then for each $p \in \D$ and each $r \in \bar f\inv(p)$ we have
    \[
        e_{\bar f}(r) =m_{f}(p) := \lcm_{q \in f\inv(p)} e_f(q).
    \]
    In particular, $\bar f$ has the same critical values as $f$.
\end{lemma}


We say a map $f : \C \to \D$ is \textbf{evenly ramified} if for all $p \in \D$ and all $q \in f\inv(p)$ we have $e_f(q) = m_f(p)$.
Lemma \ref{lemma: galois ramification} tells us that Galois maps are evenly ramified.
The following lemma provides a converse for rational functions.

\begin{lemma}
\label{lemma: Galois duck}
    A rational function $f : \PP^1 \to \PP^1$ is evenly ramified if and only if $f$ is Galois.
\end{lemma}

\begin{proof}
    We only need to prove the forward direction.
    If $f$ is evenly ramified, then so are all the conjugates of $f$.
    Let $\bar f : \C \to \PP^1$ be the Galois closure and suppose that $\bar f = fg$ is the factorization through $f$.
    Then Lemma \ref{lemma: galois ramification} implies that $g : \C \to \PP^1$ is unramified.
    The only unramified covers of $\PP^1$ are isomorphisms, hence $g$ is an isomorphism and $f$ is Galois.
\end{proof}

\subsection{Fried's Theorem}

Theorem \ref{thm: frieds thm} is equivalent to a result proved by Fried \cite[Prop. 2]{Fried}.

\begin{theorem}[Fried]
\label{thm: frieds thm}
    Let $f_i : \C_i \to \D$ for $i = 1, 2$ be maps between curves.
    There are decompositions $f_i = g_ih_i$ such that $g_1$ and $g_2$ have the same Galois closure, and the fiber product of $g_1, g_2$ has the same number of irreducible components as the fiber product of $f_1, f_2$.
\end{theorem}

Our Theorem \ref{thm: frieds thm} and Fried's \cite[Prop. 2]{Fried} are stated quite differently.
Both results may be seen as consequences of a general result about products of finite $G$-sets which we prove below.

Let $G$ be a group and let $X$ be a $G$-set; we write $gx$ to denote the action of $g \in G$ on $x \in X$.
We say $N \unlhd G$ is a \textbf{normal stabilizer} of $X$ if $N$ fixes every point in $X$.
The \textbf{Galois group} of $X$ is the maximal normal stabilizer of $X$.
If $X_1, X_2$ have the same Galois group, then we say that $X_1, X_2$ have the \textbf{same Galois closure}.
Note that if $h: X \rightarrow Y$ is a $G$-map, then every orbit of $X$ is mapped onto an orbit of $Y$, giving us a well-defined function from the orbits of $X$ to the orbits of $Y$.

\begin{theorem}
\label{thm: fried gset}
    Let $X_1, X_2$ be finite $G$-sets.
    Then for $i = 1, 2$, there are $G$-sets $Y_i$ and $G$-maps $h_i : X_i \to Y_i$ such that 
    \begin{enumerate}
        \item $Y_1, Y_2$ have the same Galois closure,
        \item $h_1 \times h_2 : X_1 \times X_2 \to Y_1 \times Y_2$ is bijective on orbits.
    \end{enumerate}
\end{theorem}

\begin{proof}
	We proceed by induction on $|X_1| + |X_2|$.
    Let $N_i$ be the Galois group of $X_i$.
    If $N_1 = N_2$, then we set $X_i = Y_i$ and $h_i = 1_{X_i}$ and we are done.
    In particular, this holds in the base case when $|X_1| + |X_2| = 2$.
    Now suppose that $|X_1| + |X_2| > 2$ and that $N_1 \not\subseteq N_2$.
    Let $X_2' := N_1X_2$ and let $h_2'' : X_2 \to X_2'$ be the natural $N_1$-orbit map.
    Then $N_1 \not\subseteq N_2$ implies $N_1$ is not a normal stabilizer of $X_2$, hence that $|X_2'| < |X_2|$.
    Then $|X_1| + |X_2'| < |X_1| + |X_2|$ and our inductive hypothesis implies that the conclusion holds for $X_1, X_2'$.
    Hence there are $Y_1, Y_2$ and maps $h_1 : X_1 \to Y_1$ and $h_2' : X_2' \to Y_2$ such that $Y_1, Y_2$ have the same Galois closure and $h_1 \times h_2'$ is bijective on orbits.
    Let $h_2 := h_2'h_2'' : X_2 \to Y_2$.
    To finish the proof it suffices to show that $1_{X_1} \times h_2''$ is bijective on orbits.

    Since $1_{X_1}$ and $h_2''$ are both surjective, $1_{X_1} \times h_2''$ is surjective on orbits. 
    Suppose $(x_1,x_2), (x_1', x_2')$ are points in $X_1 \times X_2$ whose image under $1_{X_1} \times h_2''$ lie in the same orbit. 
    Then there exists a $g\in G$ such that
	\[
		g(x_1, N_1x_2) = (x_1', N_1x_2').
	\]
	Hence $g x_1 = x_1'$ and $g N_1x_2 = N_1x_2'$. 
    So there exists $n\in N_1$ for which $gn x_2 = x_2'$. 
    Since $N_1$ is a normal stabilizer of $X_1$, we have 
	\(
		gn x_1 = g x_1 = x_1'.
	\)
	So $gn (x_1, x_2) = (x_1', x_2')$ implying that $(x_1, x_2)$ and $(x_1', x_2')$ are in the same orbit of $X_1 \times X_2$.
    Therefore $1_{X_1} \times h_2''$ is injective, hence bijective, on orbits.
\end{proof}

Theorem \ref{thm: frieds thm} is the translation of Theorem \ref{thm: fried gset} via Galois theory.
The following corollary highlights how we may use Theorem \ref{thm: frieds thm} to analyze reducible fiber products.

\begin{corollary}
\label{cor: fried}
    Let $f_i : \C_i \to \D$ for $i = 1, 2$ be maps between curves.
    If the fiber product of $f_1, f_2$ is reducible, then there are decompositions $f_i = g_ih_i$ such that $g_1, g_2$ have the same Galois closure and both have degree at least 2.
\end{corollary}

\begin{proof}
    Let $f_i = g_ih_i$ be the decompositions provided by Theorem \ref{thm: frieds thm}.
    Then $g_1, g_2$ have the same Galois closure.
    If at least one $g_i$ had degree 1, then the fiber product of $g_1, g_2$ would be irreducible.
    However, Theorem \ref{thm: frieds thm} implies that the fiber product of $g_1, g_2$ has the same number of components as the fiber product of $f_1, f_2$, which, by assumption, has at least two components.
    Thus $g_1, g_2$ must both have degree at least 2.
\end{proof}

We apply Corollary \ref{cor: fried} to prove the following lemma.

\begin{lemma}
\label{lemma: power map fiber product}
    Let $f$ be a rational function, let $d \geq 2$, and let $p, q \in \PP^1$ be distinct points.
    Let $g = \mu x^d \mu\inv$ where $\mu$ is a M\"obius transformation satisfying $\mu(0) = p$ and $\mu(\infty) = q$.
    \begin{enumerate}
        \item If $d \mid e_f(r)$ for all $r \in f\inv(\{p,q\})$, then there is a rational function $h$ such that $f = gh$.
        
        \item If there is not a factor $e > 1$ of $d$ such that $e \mid e_f(r)$ for all $r \in f\inv(\{p,q\})$, then the fiber product of $f$ and $g$ is irreducible.
    \end{enumerate}
\end{lemma}

\begin{proof}
    (1)
    Let $\C$ denote an irreducible component of the fiber product of $f$ and $g$.
    \[
    \begin{tikzcd}
    \mathbb{P}^1 \arrow[d, "g"'] & \mathcal{C} \arrow[l, "f'"'] \arrow[d, "g'"] \\
    \mathbb{P}^1                 & \mathbb{P}^1 \arrow[l, "f"]
    \end{tikzcd}
    \]
    Our assumption that $d \mid e_f(p)$ for all $p \in f\inv(\{p, q\})$, together with Abhyankar's lemma, implies that $g'$ is unramified.
    There are no unramified branched covers of $\PP^1$ with degree greater than one, hence $g' : \C \to \PP^1$ is an isomorphism.
    Thus setting $h = f'{g'}\inv$ gives us the desired decomposition.

    (2)
    We prove the contrapositive.
    Suppose that the fiber product of $f$ and $g$ is reducible.
    Then Corollary \ref{cor: fried} implies that $f = f_1f_2$ and $g = g_1g_2$ where $f_1$ and $g_1$ have the same Galois closure and both have degree bigger than one.
    Every left factor of $g$ is Galois and totally ramified over $p, q$ with ramification $e$ for some $e \mid d$.
    Hence we may assume without loss of generality that $f_1 = g_1$.
    Thus $e \mid e_{f_1}(r)$ for all $r \in f_1\inv(\{p, q\})$, and it follows that $e \mid e_f(r)$ for all $r \in f\inv(\{p, q\})$.
\end{proof}

\subsection{Riemann-Hurwitz}

If $\C$ is an irreducible curve, then the Euler characteristic of $\C$ is defined as $\chi(\C) = 2 - 2\gen(\C)$.
If $\C$ is reducible, say $\C = \C_1 \sqcup \ldots \sqcup \C_n$, then $\chi(\C) := \chi(\C_1) + \ldots + \chi(\C_n)$.
We make frequent use of the well-known Riemann-Hurwitz formula.

\begin{lemma}[Riemann-Hurwitz]
\label{lemma: R-H}
    Let $f : \C \to \D$ be a map between curves.
    Then
    \[
        \chi(\C) = \deg f \chi(\D) - \sum_{q \in \C} e_f(q) - 1.
    \]
\end{lemma}

\begin{proof}
    See \cite[Cor. 2.4]{hartshorne2013algebraic}.
\end{proof}

The following lemma records several useful consequences of Riemann-Hurwitz.

\begin{lemma}
\label{lemma: RH cor}
    Let $\C, \D$ be irreducible curves.
    \begin{enumerate}
        \item If there exists a  map $f : \C \rightarrow \D$, then $\gen(\C) \geq \gen(\D)$.
        
        \item If there exists an endomorphism $f: \D \rightarrow \D$ with $\deg f \geq 2$, then $\gen(\D) \leq 1$.
        
        \item If $\gen(\C) = \gen(\D) = 1$, then every map $f: \C \rightarrow \D$ is unramified and Galois.
    \end{enumerate}
\end{lemma}
\begin{proof}
    (1) The Riemann-Hurwitz formula implies that
    \[
        \gen(\C) - 1 = \deg f(\gen(\D)-1) + \frac{1}{2}\sum_{q \in \C}e_f(q) - 1.
    \]
    Since genus is a non-negative integer, it follows that $\gen(\C) \geq \gen(\D)$.
    
    (2) The Riemann-Hurwitz formula gives us
    \[
        (\deg f - 1)\chi(\D) = \sum_{q\in \D} e_f(q) - 1.
    \]
    Since the right hand side is non-negative and $\deg f - 1 > 0$ it follows that $\chi(\D)\geq 0$, or equivalently that $\gen(\D) \leq 1$.
    
    (3) If $\gen(\C) = \gen(\D) = 1$, then $\chi(\C) = \chi(\D) = 0$ and Riemann-Hurwitz implies that $e_f(q) = 1$ for all $q \in \C$.
    Let $q \in \C$ be a point and let $p := f(q)$.
    Using $p$ and $q$ as our base points, $\C$ and $\D$ are elliptic curves and $f$ is an isogeny.
    Every isogeny is Galois with Galois group $\ker f$; see Silverman \cite[Thm. 4.10 (c)]{Silverman_ellcurves}.
\end{proof}

\subsection{Low-genus fiber products}

Let $g$ be a rational function with degree $e \geq 2$.
Let $B_{g,d}$ for $d \geq 2$ denote the set of all maps $f : \C \to \PP^1$ of degree $d$ such that the fiber product of $f$ and $g$ is irreducible with genus at most one.
Given $p \in \PP^1$, let
\[
    m_{g,d}(p) := \sup_{f\in B_{g,d}}m_f(p) = \sup_{f\in B_{g,d}} \lcm_{q \in f\inv(p)}e_f(p).
\]
Let $V_{g,d} := \bigcup_{f\in B_{g,d}}V_f$ where $V_f$ is the set of critical values of $f$.
The following result appears in Chapter 5 of the author's Ph.D. thesis, which was joint work with Michael Zieve.
We include a proof for the reader's convenience.

\begin{theorem}[{\cite[Thm. 5.3.6]{thesis}}]
\label{thm: bdd genus}
    If $2 \leq d < e/2$, then
    \[
        \sum_{p\in V_{g,d}} 1 - \frac{1}{ m_{g,d}(p)} \leq \frac{2e - 2}{e - 2d}.
    \]
\end{theorem}
\begin{proof}
    The Riemann-Hurwitz formula implies that
    \begin{equation}
    \label{eqn: bound}
        2e - 2 = \sum_{q\in \PP^1}e_{g}(q) - 1 \geq \sum_{p\in V_{g,d}}\sum_{q \in g\inv(p)}e_{g}(q) - 1 = \sum_{p\in V_{g,d}}e - |g\inv(p)|.
    \end{equation}
    Suppose that $p \in V_{g,d}$, $q \in g\inv(p)$, and $f : \C \to \PP^1$ is an element of $B_{g,d}$.
    Let 
    \[
        \begin{tikzcd}
        \mathcal{C} \arrow[d, "f"'] & \mathcal{C}' \arrow[l] \arrow[d, "f'"] \\
        \mathbb{P}^1                & \mathbb{P}^1 \arrow[l, "g"]           
        \end{tikzcd}
    \]
    be the fiber product diagram of $f$ and $g$.
    Note that $f \in B_{g,d}$ implies that $\C'$ is irreducible with $\gen(\C') \leq 1$.
    If $e_g(q)$ is not divisible by $m_f(p)$, then Abhyankar's lemma implies that $q$ is a critical value of $f'$.
    Riemann-Hurwitz bounds the number of critical values of $f'$ by
    \[
        |V_{f'}| \leq \sum_{q\in \C'}e_{f'}(q) - 1 =  2d+ 2(\gen(\C') - 1) \leq 2d.
    \]
    Hence $m_f(p)$ divides $e_g(q)$ for all but at most $2d$ points $q \in g\inv(p)$. 
    Thus
    \[
        |g\inv(p)| \leq 2d + \frac{e - 2d}{m_f(p)}
    \]
    for all $f \in B_{g,d}$, and therefore
    \[
        |g\inv(p)| \leq 2d + \frac{e - 2d}{{m}_{g,d}(p)}.
    \]
    Hence 
    \[
        e - |g\inv(p)| \geq (e - 2d)\Big(1 - \frac{1}{{m}_{g,d}(p)}\Big).
    \]
    Thus \eqref{eqn: bound} and our assumption that $e - 2d > 0$ implies that
    \[
        2e - 2 \geq \sum_{p\in V_{g,d}}(e - 2d)\Big(1 - \frac{1}{{m}_{g,d}(p)}\Big) \Longrightarrow \sum_{p\in V_{g,d}}1 - \frac{1}{{m}_{g,d}(p)} \leq \frac{2e - 2}{e - 2d}.\qedhere
    \]
\end{proof}

Theorem \ref{thm: bdd genus} shows that $f$ having an irreducible fiber product with $g$ of low genus places significant ramification constraints on $f$.
We say a map $f : \C \to \PP^1$ from an irreducible curve $\C$ is \textbf{low-genus} if the Galois closure $\bar f : \overline\C \to \PP^1$ has $\gen(\overline\C) \leq 1$.

\begin{corollary}
\label{corollary: low genus criteria}
    Let $d \geq 2$ and suppose $g$ is a rational function with degree $e$ such that 
    \[
        e > 170d - 84,
    \]
    then
    \begin{enumerate}
        \item $\sum_{p \in V_{g,d}} 1 - \frac{1}{{m}_{g,d}(p)} \leq 2$,
        \item $V_{g,d}$ has at most 4 elements,
        \item Every $f \in B_{g,d}$ has low-genus.
    \end{enumerate}
\end{corollary}

\begin{proof}
    (1)
    The inequality $e > 170d - 84$ is equivalent to
    \[
        \frac{2e - 2}{e - 2d} < 2 + \frac{1}{42}.
    \]
    Thus Theorem \ref{thm: bdd genus} implies
    \[
        \sum_{p\in V_{g,d}}1 - \frac{1}{{m}_{g,d}(p)} < 2 + \frac{1}{42}.
    \]
    This inequality with positive integers ${m}_{g,d}(p)$ implies that
    \[
        \sum_{p\in V_{g,d}}1 - \frac{1}{{m}_{g,d}(p)} \leq 2
    \]
    by a well-known calculation (see, for example, Miranda \cite[Lem. 3.8 (c)]{Miranda}).
    
    (2) 
    Since ${m}_{g,d}(p) \geq 2$ for each $p \in V_{g,d}$, we have $1 - \frac{1}{{m}_{g,d}(p)} \geq \frac{1}{2}$. 
    Hence (1) implies $|V_{g,d}| \leq 4$.
    
    (3) 
    If $f \in B_{g,d}$ and $\bar f: \overline  \C \rightarrow \PP^1$ is the Galois closure of $f : \C \rightarrow \PP^1$, then Lemma \ref{lemma: galois ramification} implies that $m_f(p)$ is the common ramification index of $\bar f$ over each point $q \in \bar f\inv(p)$ for each $p \in \PP^1$.
    Therefore, Riemann-Hurwitz applied to $\bar f$ gives us
    \begin{align*}
        2(\gen(\overline \C) - 1) &= -2\deg \bar f + \sum_{q\in \overline \C} e_{\bar f}(q) - 1\\
        &= \deg \bar f\Big({-2} + \sum_{p \in \PP^1}1 - \frac{1}{m_f(p)}\Big)\\
        &\leq \deg\bar f\Big({-2} + \sum_{p \in V_{g,d}}1 - \frac{1}{m_{g,d}(p)}\Big)\\
        &\leq 0.
    \end{align*}
    Hence $\gen(\overline \C) \leq 1$, implying that $f$ is low-genus.
\end{proof}

\section{Low-genus maps}
\label{sec: low-genus}

For this section, let $f : \C \to \D$ be a map between irreducible curves.
The \textbf{ramification type} of $f$ is the unordered list of partitions recording the ramification over each critical value of $f$.
For example, the ramification type $((1^12^2), (1^12^2), (5^1))$ corresponds to a degree 5 map with three critical values: two which have a single unramified pre-image and two pre-images with ramification index 2, and one critical value with a totally ramified pre-image.
The \textbf{signature} of $f$ is the unordered list $(m_f(p_1), m_f(p_2),\ldots, m_f(p_n))$ where the $p_i$ are the critical values of $f$.
Note that the signature of $f$ does not include the critical values of $f$.
Lemma \ref{lemma: galois ramification} implies that $f$ and its Galois closure $\bar f : \overline{\C} \to \D$ have the same signature.
The main goal of this section is to provide a census of all ramification types of low-genus maps.

Riemann-Hurwitz simplifies for Galois maps with Galois group $G$ to
\[
    \chi(\C) = |G|\chi(\D) - \sum_{q \in \overline\C}e_{\bar f}(q) - 1
= |G|\chi(\D) - |G|\sum_{p \in \D} \Big(1 - \frac{1}{m_f(p)}\Big),
\]
or equivalently
\begin{equation}
\label{eqn: R-H galois}
     \chi(\D) - \frac{\chi(\overline\C)}{|G|} = \sum_{p\in \D} \Big(1 - \frac{1}{m_f(p)}\Big).
\end{equation}
If $f$ has low-genus, then $\C$, $\overline\C$, and $\D$ all have genus at most one by Lemma \ref{lemma: RH cor}.
If $\D$ has genus one, then Lemma \ref{lemma: RH cor} implies that $f = \bar f$, that $\C = \overline\C$ also has genus one, and that $f$ is unramified.
Thus in all other cases, $\D \cong \PP^1$ has genus zero and $\chi(\D) = 2$.
In the remainder of this section we classify all the low-genus maps based on the genus of $\C$ and $\overline\C$.

\subsection{Galois closure genus zero}
If $\overline\C$ has genus zero, then \eqref{eqn: R-H galois} becomes
\[
    2\Big(1 - \frac{1}{|G|}\Big) = \sum_{p\in \PP^1}  \Big(1 - \frac{1}{m_f(p)}\Big).
\]
The numerical solutions to this equation in positive integers fall into two infinite families and three exceptional cases.
Each solution is realized by a unique conjugacy class of finite subgroups of $\aut(\PP^1)$ (i.e. groups of M\"obius transformations).
These signatures and the associated groups are listed in the table below.

\begin{center}
    \begin{tabular}{|c|c|c|}
    \hline
        $G$ & order & signature \\
    \hline
        $C_n$ & $n$ & $(n,n)$\\
        $D_n$ & $2n$ & $(2,2,n)$\\
        $A_4$ & 12 & $(2,3,3)$\\
        $S_4$ & 24 & $(2,3,4)$\\
        $A_5$ & 60 & $(2,3,5)$\\
    \hline
    \end{tabular}
\end{center}

We refer to the $(n,n)$ signatures as \textbf{cyclic}, the $(2,2,n)$ signatures as \textbf{dihedral}, and the other three as \textbf{Platonic}, owing to their connection with the rotational symmetries of Platonic solids.
Collectively we refer to these as \textbf{genus zero signatures}.
We say two maps $f_1 : \C_1 \to \D_1$ and $f_2 : \C_2 \to \D_2$ are \textbf{essentially equivalent} if there are isomorphisms $\mu_1 : \D_1 \to \D_2$ and $\mu_2 : \C_1 \to \C_2$ such that $f_2 = \mu_1 f_1 \mu_2\inv$.
The uniqueness up to conjugacy of the finite groups acting on $\PP^1$ implies that rational functions with a genus zero Galois closure are determined up to essential equivalence by their ramification types.

\begin{lemma}
\label{lemma: gen 0 exceptional rigidity}
    If $f$ and $g$ are rational functions with the same ramification type and genus zero signatures, then $f$ and $g$ are essentially equivalent.
\end{lemma}

\begin{proof}
    Let $G_f$ and $G_g$ be the Galois groups of the Galois closures $\bar f, \bar g$ of $f, g$, respectively.
    As discussed, the signature of $f$ determines $G_f$ up to conjugacy, and our assumption that $f$ and $g$ have the same ramification types implies that there is a M\"obius transformation $\mu$ such that $G_f = \mu G_g \mu\inv$.
    Thus $\bar f$ and $\bar g \mu\inv$ both have Galois group $G_f$.
    Note that $\bar g \mu\inv $ is also a Galois closure of $g$, so we may suppose without loss of generality that $\mu$ is the identity and that $G := G_f$ is the Galois group of both $\bar f$ and $\bar g$.

    Let $\bar f = ff'$ be a factorization of $\bar f$ through $f$.
    Then the ramification type of $f'$ is determined by that of $f$ and $\bar f$.
    The rational function $f'$ is Galois with Galois group $H_f \subseteq G$.
    The isomorphism class of $H_f$ is determined by the ramification type of $f'$, hence by that of $f$.
    Therefore our assumption that $f$ and $g$ have the same ramification type implies that the subgroups $H_f, H_g \subseteq G$ are conjugate in the group of all M\"obius transformations.
    If $G$ is isomorphic to $C_n$, $A_4$, $S_4$, $A_5$, or $D_n$ with $n$ odd, then $G$ has a unique conjugacy class of subgroups of each given order.
    If $G$ is isomorphic to $D_n$ with $n$ even, then $G$ has two conjugacy classes of reflections which are related by an outer automorphism; this outer automorphism can be realized by conjugating $G$ by a M\"obius transformation.
    Thus after potentially replacing $\bar g$ by some $\bar g \mu\inv$ we can assume that $\bar f$ and $\bar g$ have the same Galois group and that $H := H_f = H_g$.
    Then we have factorizations $\bar f = ff'$ and $\bar g = gg'$ where $f'$ and $g'$ have the same Galois group $H$.
    Therefore there is a M\"obius transformation $\mu_2$ such that $f' = \mu_2 g'$.
    Since $\bar f$ and $\bar g$ have the same Galois group, there is a M\"obius transformation $\mu_1$ such that $\bar f = \mu_1 \bar g$.
    Thus
    \[
        ff' = \bar f = \mu_1 \bar g = \mu_1 g g' = \mu_1 g\mu_2\inv \mu_2 g' = \mu_1g\mu_2\inv f'.
    \]
    Since $f'$ is non-constant, it follows that $f = \mu_1 g \mu_2\inv$.
\end{proof}

Lemma \ref{lemma: gen 0 exceptional rigidity} reduces the classification of all rational functions with genus zero Galois closure up to essential equivalence to determining their ramification types; this is a finite calculation depending only on the groups and how they act on $\PP^1$
We collect the results of this calculation in Table \ref{table: gen 0 census} with ramification types grouped by signature.
The maps with cyclic and dihedral signature are well-known.
For example, $x^d$ has the cyclic ramification type $((d^1),(d^1))$.
The rational function $\frac{x^d + x^{-d}}{2}$ has the Galois dihedral ramification type $((2^d), (2^d), (d^2))$.
Recall that the $d$th Chebyshev polynomial $T_d(x)$ is defined by the functional equation
\[
    T_d(\tfrac{x + x\inv}{2}) = \tfrac{x^d + x^{-d}}{2},
\]
which is equivalent to the following commutative diagram.
\[
    \begin{tikzcd}
    \mathbb{P}^1 \arrow[d, "\frac{x + x^{-1}}{2}"'] & \mathbb{P}^1 \arrow[l, "x^d"'] \arrow[d, "\frac{x + x^{-1}}{2}"] \\
    \mathbb{P}^1                                    & \mathbb{P}^1 \arrow[l, "T_d"]                                   
    \end{tikzcd}
\]
Note that we have opted for a slightly different normalization of our Chebyshev polynomials than is standard.
The upshot of this normalization is that the critical points of $T_d$ are $\{\pm 1, \infty\}$ which contains (and typically equals) the set of critical values of $T_d$.

\begin{table}[h!]
    \begin{tabular}{|c|c|c|c|}
        \hline
        signature & $G$ & degree & ramification type \\
        \hline
        $(d,d)$ & $C_d$ & $d$ & $((d^1), (d^1))$ \\
        \hline
        $(2,2,d)$ & $D_d$ & $2d$ & $((2^d), (2^d), (d^2))$\\
        & if $d = 2k + 1$ & $d$ & $((1^12^k), (1^12^k), (n^1))$\\
        & if $d = 2k\phantom{+1\,\,\,}$ & $d$ & $((1^22^{k-1}), (2^k),(n^1))$\\
        \hline
        $(2,3,3)$ & $A_4$ & 12 & $((2^6),(3^4),(3^4))$\\
              & & 6 & $((1^22^2),(3^2),(3^2))$\\
              & & 4 & $((1^13^1),(2^2),(1^13^1))$\\
        \hline
        $(2,3,4)$ & $S_4$ & 24 & $((2^{12}),(3^8),(4^6))$\\
              & & 12 & $((1^22^5),(3^4),(4^3))$\\
              & & 12 & $((2^6),(3^4),(2^24^2))$\\
              & & 8  & $((2^4),(1^23^2),(4^2))$\\
              & & 6  & $((2^3),(3^2),(1^24^1))$\\
              & & 6  & $((1^22^2),(3^2),(2^14^1))$\\
              & & 4  & $((1^22^1),(1^13^1),(4^1))$\\
        \hline
        $(2,3,5)$ & $A_5$ & 60 & $((2^{30}),(3^{20}),(5^{12}))$\\
              & & 30 & $((1^22^{14}),(3^{10}),(5^6))$\\
              & & 20 & $((2^{10}),(1^23^6),(5^4))$\\
              & & 15 & $((1^32^6),(3^5),(5^3))$\\
              & & 12 & $((2^6),(3^4),(1^25^2))$\\
              & & 10 & $((1^22^4),(1^13^3),(5^2))$\\
              & & 6  & $((1^22^2),(3^2),(1^15^1))$\\
              & & 5  & $((1^12^2),(1^23^1),(5^1))$\\
        \hline
    \end{tabular}
    \caption{Rational functions with genus 0 Galois closure.}
    \label{table: gen 0 census}
\end{table}

\subsection{Genus one Galois closure}
\label{sec: genus one galois}

If $\overline\C$ has genus one, then \eqref{eqn: R-H galois} further simplifies to
\[
    2 = \sum_p \Big(1 - \frac{1}{m_f(p)}\Big).
\]
This equation has only four solutions, namely $(2,2,2,2), (3,3,3), (2,4,4),$ and $(2,3,6)$.
Each of these signatures may occur with infinitely many non-isomorphic Galois groups.

\subsubsection{Cyclic quotients and isogenies}
First suppose that $\C$ also has genus one.
Lemma \ref{lemma: canonical factorization} shows that such maps $f$ have a canonical factorization through a cyclic quotient.

\begin{lemma}
\label{lemma: canonical factorization}
    Let $f : \C \to \PP^1$ be a map from an irreducible genus one curve $\C$ and suppose that the Galois closure of $f$ has genus one.
    Then $f$ factors as $f = \pi \psi$ where $\psi : \C \to \C'$ is a map between genus one curves and $\pi : \C' \to \PP^1$ is a cyclic quotient.
\end{lemma}

\begin{proof}
    Let $\bar f : \overline{\C} \to \PP^1$ be the Galois closure of $f$ and suppose that $\bar f = f f'$ is a factorization of $\bar f$ through $f$.
    The Galois group $G$ of $\bar f$ is a finite automorphism group of $\overline{\C}$, hence is a semidirect product of a group $N$ of translations by torsion points and a cyclic group $H$ of elliptic curve automorphisms.
    Let $N'\subseteq G$ be the Galois group of $f': \overline{\C} \to \C$.
    Lemma \ref{lemma: RH cor} implies $f'$ is unramified, hence $N' \subseteq N$.
    Since $N$ is normal in $G$, the quotient $\C' := \overline{\C}/N$ has genus one and the left factor $\pi : \C' \to \PP^1$ of $\bar f$ is Galois with cyclic Galois group $G/N \cong H$. 
    The containment $N' \subseteq N$ implies that $\pi$ is a left factor of $f$, hence that $f = \pi \psi$ for some $\psi : \C \to \C'$.
\end{proof}

There are essentially four possible cyclic quotients $\pi$ as described in Lemma \ref{lemma: canonical factorization}, one for each signature.
The quotient with signature $(2,2,2,2)$ exists for any genus one curve and the configuration of critical values of $\pi$ up to M\"obius transformation determine the genus one curve $\C$ up to isomorphism.
The quotients corresponding to the other three signatures are rigid and only exist for particular genus one curves with extra symmetries.
Table \ref{table: genus 1 to 0 galois} records the degree of these cyclic quotients and the $j$-invariants of the curves they correspond to.

\begin{table}[h]
    \centering
    \begin{tabular}{|c|c|c|}
    \hline
       signature  & degree & $j$-invariant \\
    \hline
        $(2,2,2,2)$ & 2 & Any \\
        $(3,3,3)$ & 3 & 0 \\
        $(2,3,6)$ & 6 & 0 \\
        $(2,4,4)$ & 4 & 1728\\
    \hline        
    \end{tabular}
    \caption{Cyclic quotients $\pi : \C \to \PP^1$ of genus one curves.}
    \label{table: genus 1 to 0 galois}
\end{table}

Let $\E_0, \E_1$ be complex elliptic curves (genus one curves with a chosen base point $\O$ which serves as the identity for the group law).
Any map $\psi : \E_1 \to \E_0$ may be expressed as $\psi = \phi + \gamma$ where $\phi$ is an isogeny and $\gamma \in \E_0$ is a point.
Recall that any complex elliptic curve $\E$ is analytically isomorphic to $\CC/\Lambda$ where $\Lambda \subseteq \CC$ is a lattice.
The isomorphism class of $\E$ is determined by $\Lambda$ up to homothety.
Let $\Lambda_i \subseteq \CC$ be a lattice such that $\E_i \cong \CC/\Lambda_i$.
The isogenies $\phi : \E_1 \to \E_0$ correspond to complex numbers $\alpha$ such that $\alpha \Lambda_1 \subseteq \Lambda_0$.
The degree of $\psi = \phi + \gamma$ is the index $[\Lambda_0 : \alpha \Lambda_1]$.
For more on isogenies of elliptic curves, see Silverman \cite[III.4]{Silverman_ellcurves}.

The automorphism group of a complex elliptic curve can only be cyclic of order 2, 4, or 6, with the latter two cases only occurring for elliptic curves with $j$-invariant 1728 and 0, respectively \cite[Thm. 10.1]{Silverman_ellcurves}.
If $\E = \CC/\Lambda$, then automorphisms of $\E$ may be expressed as $\zeta x \bmod \Lambda$ for some root of unity $\zeta$.
The following lemma characterizes the critical points of each of the cyclic quotients $\pi$.
We will use this to classify the ramification types of generalized Latt\'es maps in the subsequent section.

\begin{lemma}
\label{lemma: critical points of cyclic quotients}
    Let $\pi : \E \to \PP^1$ be a quotient by $\zeta x \bmod \Lambda$ for some root of unity $\zeta$.
    \begin{enumerate}
        \item If $\zeta = -1$, then $\pi$ has signature $(2,2,2,2)$ and the critical points of $\pi$ are the 2-torsion points on $\E$.

        \item If $\zeta = \omega$ is a third root of unity, then $\pi$ has signature $(3,3,3)$, and the critical points of $\pi$ are the $(\omega - 1)$-torsion points on $\E$.

        \item If $\zeta = -\omega$ is a 6th root of unity, then $\pi$ has signature $(2,3,6)$.
        The identity $\O$ is a critical point of $\pi$ with ramification index 6; the two primitive $(\omega - 1)$-torsion points have ramification index 3; and the three primitive 2-torsion points have ramification index 2.

        \item If $\zeta = i$ is a 4th root of unity, then $\pi$ has signature $(2,4,4)$.
        The $(i - 1)$-torsion points have $\pi$-ramification index 4 and the two primitive 2-torsion points have ramification index 2.
    \end{enumerate}
\end{lemma}

\begin{proof}
    (1)
    The map $-x \bmod \Lambda$ is an order 2 automorphism of any elliptic curve $\CC/\Lambda$.
    The critical points of the quotient $\pi$ are those points which are fixed by $-x$.
    If $-\gamma \equiv \gamma \bmod \Lambda$, then $2\gamma \equiv 0 \bmod \Lambda$, hence $\gamma$ is a 2-torsion point.
    Each of these four points has ramification index 2.
    Hence the signature of $\pi$ is $(2,2,2,2)$.

    (2)
    If $j(\E) = 0$, then $\E \cong \CC/\Lambda$ where $\Lambda = \ZZ[\omega]$ and $\omega$ is a third root of unity. 
    The automorphism group of $\E$ is cyclic of order 6 generated by $-\omega x \bmod \Lambda$.
    Suppose $\pi$ is the quotient of $\E$ by $\omega x \bmod \Lambda$.
    Then $\pi$ has degree 3 and is totally ramified at the fixed points of $\omega x \bmod \Lambda$, which are precisely the $(\omega - 1)$-torsion points of $\E$.
    Hence the signature of $\pi$ is $(3,3,3)$.

    (3)
    Suppose now that $\pi$ is the quotient of the $j(\E) = 0$ curve by $\sigma(x) := -\omega x \bmod \Lambda$.
    Since $-\omega - 1 = \omega^2$ is a unit, the only fixed point of $\sigma$ is the identity, hence $\pi$ is totally ramified at $\O$.
    The fixed points of $\sigma^2 = \omega^2 x \bmod \Lambda$ are the $(\omega - 1)$ torsion points, hence the two primitive $(\omega - 1)$-torsion points each have ramification index 3 and $\sigma$ acts transitively on these two points.
    The fixed points of $\sigma^3 = -x \bmod \Lambda$ are the 2-torsion points, hence the primitive 2-torsion points each have ramification index 3 and $\sigma$ acts transitively on these three points.
    Thus the signature of $\pi$ is $(2,3,6)$.

    (4)
    Finally suppose that $j(\E) = 1728$, hence $\E \cong \CC/\Lambda$ where $\Lambda = \ZZ[i]$.
    Let $\sigma(x) := ix \bmod \Lambda$.
    The fixed points of $\sigma$ are the two $(i - 1)$-torsion points, which each have ramification index 4 under $\pi$.
    The fixed points of $\sigma^2 = -x \bmod \Lambda$ are the 2-torsion points, hence the two primitive 2-torsion points have ramification index 2 under $\pi$ and $\sigma$ acts transitively on these points.
    Therefore the signature of $\pi$ is $(2,4,4)$.
\end{proof}

\subsubsection{Generalized Latt\'es maps}
\label{sec: lattes}

A \textbf{generalized Latt\'es map} is a rational function $f$ for which there exists elliptic curves $\E_i$, cyclic quotients by elliptic curve automorphisms $\pi_i : \E_i \to \PP^1$, and a map $\psi : \E_1 \to \E_0$ such that the following diagram commutes.
\begin{equation}
\label{eqn: gen lattes def}
    \begin{tikzcd}
    \mathcal{E}_0 \arrow[d, "\pi_0"'] & \mathcal{E}_1 \arrow[l, "\psi"'] \arrow[d, "\pi_1"] \\
    \mathbb{P}^1                      & \mathbb{P}^1 \arrow[l, "f"]     
    \end{tikzcd}
\end{equation}
The diagram \eqref{eqn: gen lattes def} uniquely determines $f$ and we write $L_{\pi_0, \pi_1, \psi} := f$.
Note that 
\[
    L_{\pi_0, \pi_1, \psi_0}L_{\pi_1,\pi_2,\psi_1} = L_{\pi_0, \pi_2, \psi_0\psi_1}.
\]
The following lemma provides a simple way to recognize generalized Latt\'es maps.

\begin{lemma}
\label{lemma: normalizing lattes}
    Let $f$ be a rational function.
    Then $f$ is a generalized Latt\'es map if and only if there is some rational function $g$ such that $fg$ has a genus one Galois closure.
\end{lemma}

\begin{proof}
    First we prove necessity.
    Let $f = L_{\pi_0, \pi_1, \psi}$.
    The Galois closure of $\pi_0\psi$ has genus one, hence the Galois closure of $f$ must have genus at most one.
    For $d \geq 2$, let $L_{\pi_1, \pi_1, [d]}$ be the standard Latt\'es map associated to the multiplication-by-$d$ endomorphism of $\E_1$.
    The Galois closure of every $L_{\pi_1, \pi_1, [d]}$ has genus one and $fL_{\pi_1, \pi_1, [d]}$ is also a generalized Latt\'es map.
    Thus $fL_{\pi_1, \pi_1, [d]}$ has a genus one Galois closure.

    Next we prove sufficiency.
    Let $h : \E_2 \to \PP^1$ be the Galois closure of $fg$ where $\E_2$ has genus one.
    Lemma \ref{lemma: canonical factorization} implies that $h$ factors as
    \[
        h: \E_2 \xrightarrow{\bar\psi} \E_0 \xrightarrow{\pi_0} \PP^1
    \]
    where $\E_0$ is an elliptic curve and $\pi_0$ is a quotient by elliptic curve automorphisms.
    Let $\E_1$ be the irreducible component of the fiber product of $f$ with $\pi_0$ through which $h$ factors.
    Lemma \ref{lemma: RH cor} implies that $\E_1$ also has genus one.
    Let $\psi : \E_1 \to \E_0$ be the map through which $\bar\psi$ factors and let $\pi_1 : \E_1 \to \PP_1$ be the restriction of the fiber product projection to $\E_1$, so that $\pi_0 \psi = f \pi_1$.
    Since $\pi_0$ is Galois with cyclic Galois group, the second group isomorphism theorem implies the same is true for $\pi_1$.
    Therefore $f = L_{\pi_0,\pi_1,\psi}$ is a generalized Latt\'es map.
    Note that if we are willing to change the base points on $\E_0, \E_1$, we may assume that $\psi$ is an isogeny.
\end{proof}

Returning to our classification of low-genus maps, suppose that $f : \C \to \PP^1$ is a map such that the Galois closure $\overline{\C}$ has genus one and $\C \cong \PP^1$ has genus zero.
Then Lemma \ref{lemma: normalizing lattes} implies that $f$ must be a generalized Latt\'es map.
Proposition \ref{prop: gen lattes ram types} classifies all the possible ramification types for generalized Latt\'es maps.

\begin{proposition}
\label{prop: gen lattes ram types}
    If $f = L_{\pi_0, \pi_1, \psi}$ is a generalized Latt\'es map with degree $d$, then the ramification type of $f$ is one of the possibilities listed in Table \ref{table: genus 1 gal closure}.
\end{proposition}

\begin{table}[h]
    \centering
    \begin{tabular}{|c||c|c|l|l|}
        \hline
            & $\pi_0$ sign. & $\pi_1$ sign. & constraint & ramification type \\
        \hline
            1 & (2,2,2,2) & (2,2,2,2) & $d \equiv 1 \bmod 2$ & $((1^1 2^{\frac{d-1}{2}}),\, (1^1 2^{\frac{d-1}{2}}),\, (1^1 2^{\frac{d-1}{2}}),\, (1^1 2^{\frac{d-1}{2}}))$ \\
            2 &           & & $d \equiv 0 \bmod 2$ & $((1^22^\frac{d-2}{2}),\, (1^22^\frac{d-2}{2}),\, (2^\frac{d}{2}),\,(2^\frac{d}{2}))$ \\
            3 &           & & $d \equiv 0 \bmod 4$ & $((1^4 2^\frac{d-4}{2}),\, (2^\frac{d}{2}),\,(2^\frac{d}{2}),\, (2^\frac{d}{2}))$ \\
        \hline
            4 & (3,3,3)   & (3,3,3) & $d \equiv 1 \bmod 3$ & $((1^13^\frac{d-1}{3}),\, (1^13^\frac{d-1}{3}),\, (1^13^\frac{d-1}{3}))$ \\
            5 &           & & $d \equiv 0 \bmod 3$ & $((1^33^\frac{d-3}{3}),\, (3^\frac{d}{3}),\, (3^\frac{d}{3}))$ \\
        \hline
            6 & (2,3,6)   & (2,3,6) & $d \equiv 1 \bmod 6$ & $((1^12^\frac{d-1}{2}),\, (1^13^\frac{d-1}{3}),\, (1^16^\frac{d-1}{6}))$\\
            7 &           & & $d \equiv 4 \bmod 6$ & $((2^\frac{d}{2}),\, (1^13^\frac{d-1}{3}),\, (1^13^16^\frac{d-4}{6}))$\\
            8 &           & & $d \equiv 3 \bmod 6$ & $((1^12^\frac{d-1}{2}),\, (3^\frac{d}{3}),\, (1^12^16^\frac{d-3}{6}))$\\
            9 &           & & $d \equiv 0 \bmod 6$ & $((2^\frac{d}{2}),\, (3^\frac{d}{3}),\, (1^12^13^16^\frac{d-6}{6}))$\\
        \cline{3-5}
            10 &          & (3,3,3) & $d \equiv 2 \bmod 6$ & $((2^\frac{d}{2}),\, (1^23^\frac{d-2}{3}),\, (2^16^\frac{d-2}{6}))$\\
            11 &          & & $d \equiv 0 \bmod 6$ & $((2^\frac{d}{2}),\, (3^\frac{d}{3}),\, (2^36^\frac{d-6}{6}))$\\
        \cline{3-5}
            12 &          & (2,2,2,2) & $d \equiv 3 \bmod 6$ & $((1^32^\frac{d-3}{2}),\, (3^\frac{d}{3}),\, (3^16^\frac{d-3}{6}))$\\
            13 &          & & $d \equiv 0 \bmod 6$ & $((1^22^\frac{d-2}{2}),\, (3^\frac{d}{3}),\, (3^26^\frac{d-6}{6}))$\\
            14 &          & & $d \equiv 0\bmod 12$  & $((2^\frac{d}{2}),\, (3^\frac{d}{3}),\, (3^46^\frac{d-12}{6}))$\\
        \hline
            15 & (2,4,4)  & (2,4,4) & $d \equiv 1 \bmod 4$ & $((1^12^\frac{d-1}{2}),\, (1^14^\frac{d-1}{4}),\, (1^14^\frac{d-1}{4}))$ \\
            16 &          & & $d \equiv 2 \bmod 4$ & $((2^\frac{d}{2}),\, (1^24^\frac{d-2}{4}),\, (2^14^\frac{d-2}{4}))$ \\
            17 &          & & $d \equiv 0 \bmod 4$ & $((2^\frac{d}{2}),\, (4^\frac{d}{4}),\, (1^22^14^\frac{d-4}{4}))$ \\
        \cline{3-5}
            18 &          & (2,2,2,2) & $d \equiv 2 \bmod 4$ & $((1^22^\frac{d-2}{2}),\, (2^14^\frac{d-2}{4}),\, (2^14^\frac{d-2}{4}))$ \\
            19 &          & & $d \equiv 0 \bmod 4$ & $((2^\frac{d}{2}),\, (2^24^\frac{d-4}{4}),\, (2^24^\frac{d-4}{4}))$ \\
            20 &          & & $d \equiv 0 \bmod 4$ & $((2^\frac{d}{2}),\, (4^\frac{d}{4}),\, (2^44^\frac{d-8}{4}))$ \\
        \hline
    \end{tabular}
    \caption{Ramification types for generalized Latt\'es maps}
    \label{table: genus 1 gal closure}
\end{table}

\begin{proof}
    Let $\pi_0, \pi_1, \psi, \E_0, \E_1$ be as in \eqref{eqn: gen lattes def}.
    As noted at the end of the proof of Lemma \ref{lemma: normalizing lattes}, by changing base points on the $\E_i$ we may assume that $\psi$ is an isogeny.
    Since $\psi$ is unramified, the signature of $f\pi_1$ is the same as that of $\pi_0$ and the ramification type of $f = L_{\pi_0, \pi_1, \psi}$ is determined by that of $\pi_0, \pi_1$ and how $\psi$ maps the critical points of $\pi_1$ to those of $\pi_0$.
    We split the classification into cases based on the signature of $\pi_0$.

    (1) \textbf{$\pi_0$ has signature $(2,2,2,2)$.}
        The signature of $\pi_1$ must also be $(2,2,2,2)$ in this case.
        Lemma \ref{lemma: critical points of cyclic quotients} implies that the critical points of $\pi_i$ are the points in $\E_i[2]$.
        Note that $|\ker\psi \cap \E_1[2]|$ is either 1, 2, or 4, and these cases translate to rows 1, 2, 3 of Table \ref{table: genus 1 gal closure}.
        For example, if the intersection is trivial, that implies $\psi$ maps $\E_1[2]$ bijectively onto $\E_0[2]$.
        Thus if $p \in \E_1[2]$, then the unique unramified $f$-preimage of $\pi_0\psi(p)$ is $\pi_1(p)$.
        On the other hand, if the intersection has order 2, then $\psi$ maps $\E_1[2]$ two-to-one onto $\E_0[2]$.
        Hence there are two critical values of $f$ above which $f$ is evenly ramified, and the other two critical values each have exactly two unramified $f$-preimages.
        In all cases, the order of these intersections suffices to determine the ramification type of $f$.

    (2) \textbf{$\pi_0$ has signature $(3,3,3)$.}
        The signature of $\pi_1$ must also be $(3,3,3)$.
        Lemma \ref{lemma: critical points of cyclic quotients} implies that the critical points of $\pi_i$ are the points in $\E_i[\omega - 1]$.
        Since $|\ker\psi \cap \E_1[\omega - 1]|$ is either 1 or 3, the ramification for $f$ corresponds to rows 4 and 5 of Table \ref{table: genus 1 gal closure}.

    (3) \textbf{$\pi_0$ has signature $(2,3,6)$.}
        There are three possibilities for the signature of $\pi_1$, namely $(2,3,6)$, $(3,3,3)$, or $(2,2,2,2)$.
        \begin{enumerate}[label=(\roman*)]
            \item \textbf{$\pi_1$ has signature $(2,3,6)$.}
                Lemma \ref{lemma: critical points of cyclic quotients} tells us that the critical points of $\pi_1$ are $\E_1[2] \cup \E_1[\omega -1]$.
                Hence the ramification of $f$ is determined by 
                \[
                    \ker\psi \cap (\E_1[2] \cup \E_1[\omega -1]) = (\ker \psi \cap \E_1[2]) \cup (\ker \psi \cap \E_1[\omega - 1]).
                \]
                In this case $\psi$ is a $\ZZ[\omega]$-linear map, hence these intersections are also $\ZZ[\omega]$-modules.
                Since $\E_1[2]$ is a cyclic $\ZZ[\omega]$-module and 2 is inert in $\ZZ[\omega]$ the order of the intersection $|\ker \psi \cap \E_1[2]|$ can only be 1 or 4.
                The intersection $|\ker\psi \cap \E_1[\omega - 1]|$ has order either 1 or 3.
                This gives rise to 4 cases corresponding to rows 6 through 9 of Table \ref{table: genus 1 gal closure}.
            \item \textbf{$\pi_1$ has signature $(3,3,3)$.}
                Lemma \ref{lemma: critical points of cyclic quotients} implies that the critical points of $\pi_1$ are $\E_1[\omega - 1]$.
                Thus $|\ker \psi \cap \E_1[\omega - 1]|$ is either 1 or 3.
                These two cases correspond to rows 10 and 11 of Table \ref{table: genus 1 gal closure}.
            \item \textbf{$\pi_1$ has signature $(2,2,2,2)$.}
                Lemma \ref{lemma: critical points of cyclic quotients} implies that the critical points of $\pi_1$ are $\E_1[2]$.
                Thus $|\ker\psi \cap \E_1[2]|$ is either 1, 2, or 4, leading to rows 12, 13, 14 of Table \ref{table: genus 1 gal closure}.
        \end{enumerate}

    (4) \textbf{$\pi_0$ has signature $(2,4,4)$.}
        The possible signatures for $\pi_1$ are $(2,4,4)$ or $(2,2,2,2)$.
        In either case, Lemma \ref{lemma: critical points of cyclic quotients} implies that the critical points of $\pi_1$ are the elements of $\E_1[2]$.
        Thus $|\ker\psi \cap \E_1[2]|$ is either 1, 2, or 4.
        This gives us rows 15, 16, 17 when $\pi_1$ has signature $(2,4,4)$ and 18, 19, 20 when $\pi_1$ has signature $(2,2,2,2)$.
\end{proof}

\section{Correspondences}
\label{sec: correspondences}

Let $\D_1, \D_2$ be curves.
A \textbf{correspondence} $F : \D_1 \rightsquigarrow \D_2$ is a curve $\C_F \subseteq \D_1 \times \D_2$ such that the restriction to $\C_F$ of the coordinate projections are finite and surjective.
We say the \textbf{degree} of $F$ is $(d,e)$ if the projections of $\C_F$ to $\D_1$ and $\D_2$ have degrees $d$ and $e$, respectively. 
If $G : \D_2 \rightsquigarrow \D_3$ is another correspondence, then the composition $GF$ is defined by taking the fiber product of the two projections to $\D_2$ and then composing as in the following diagram.

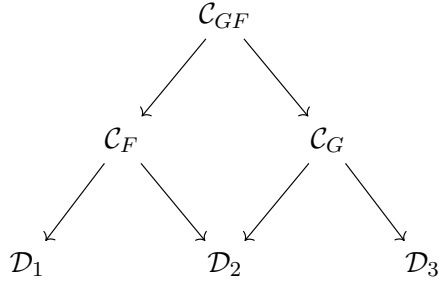
\begin{center}
    \begin{tikzcd}[row sep=large, column sep=small]
        & & \mathcal{C}_{GF} \arrow[dl] \arrow[dr] & & \\
        & \mathcal{C}_F \arrow[dl] \arrow[dr] & & \mathcal{C}_G \arrow[dl] \arrow[dr] & \\
        \mathcal{D}_1 & & \mathcal{D}_2 & & \mathcal{D}_3
    \end{tikzcd}
\end{center}

Equivalently, $\C_{GF}$ is the normalization of the curve 
\[
    \{(p,r) \in \D_1 \times \D_3 : \exists\, q \in \D_2, (p,q) \in \C_F \text{ and } (q,r) \in \C_G\}.
\]
If $\C$ is a curve together with finite maps $f : \C \to \D_1$ and $g : \C \to \D_2$, then we say this pair of maps \textbf{induces} the correspondence $F$ where $\C_F$ is the normalization of the image curve $(f(t), g(t))$.
Note that $\C \cong \C_F$ precisely when $f$ and $g$ have no common right factors of degree at least two.
This is automatic if, for example, $\deg f$ and $\deg g$ are coprime.

The \textbf{transpose} of $F : \D_1 \rightsquigarrow \D_2$ is the correspondence $F^* : \D_2 \rightsquigarrow \D_1$ where $\C_{F^*}$ is the image of $\C_F$ under the map $\D_1 \times \D_2 \to \D_2 \times \D_1$ which swaps coordinates.
Note that $(GF)^* = F^*G^*$.

Suppose now that $F: \D \rightsquigarrow \D$ is a correspondence from $\D$ to itself and let $f_1, g_1 : \C_F \to \D$ denote the coordinate projections.
Let $F^n$ denote the $n$-fold composition of $F$ with itself.
We recursively define $f_{n+1}, g_{n+1}$ to be the maps in the following fiber product diagram where $\C_{F^0} := \D$.
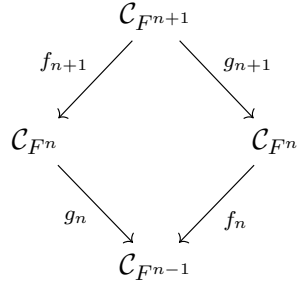
\begin{center}
    \begin{tikzcd}[row sep=large, column sep=small]
        & & \C_{F^{n+1}} \arrow[dl, "f_{n+1}"'] \arrow[dr, "g_{n+1}"] & & \\
        & \C_{F^{n}}  \arrow[dr, "g_{n}"'] & & \C_{F^{n}} \arrow[dl, "f_n"] & \\
        & & \C_{F^{n-1}} & &
    \end{tikzcd}
\end{center}
If $1 \leq m \leq n$, define $f_{m,n} : \C_{F^n} \to \C_{F^{m-1}}$ to be the composition $f_{m,n} := f_m f_{m+1}\cdots f_n$; define $g_{m,n}$ similarly.
Note that the fiber product of $f_{m,n_1}$ and $g_{m,n_2}$ is $\C_{F^{n_1 + n_2}}$.
We say $F : \D \rightsquigarrow \D$ is \textbf{stable} if $\C_{F^n}$ is irreducible for all $n \geq 0$.

If $F, F': \D \rightsquigarrow \D$ are correspondences, with projections $f_1, g_1 : \C_F \to \D$ and $f_1', g_1' : \C_{F'} \to \D$, then we say $F$ and $F'$ are \textbf{isomorphic} and write $F \cong F'$ if there is an automorphism $\mu_1 : \D \to \D$ and an isomorphism $\mu_2 : \C_{F'} \to \C_F$ such that $f_1 = \mu_1 f_1' \mu_2\inv$ and $g_1 = \mu_1 g_1' \mu_2\inv$.
Note that $F \cong F'$ implies that $F^n \cong {F'}^n$ for all $n \geq 0$.

If $F : \D \rightsquigarrow \D$ is a correspondence and $p \in \D$ is a point, then we let
\[
    F^n(p) := \{q \in \D : \exists\, r \in \C_{F^n}, f_{1,n}(r) = p, g_{1,n}(r) = q\}
\]
denote the set of all possible $n$th iterates of $p$ under $F$.

\begin{example}
    Consider the correspondence $F : \PP^1 \rightsquigarrow \PP^1$ defined by the projective closure of $y^2 = x^3 - 1$.
    We may view $F$ as the multi-valued algebraic function 
    \[
        F(x) := \pm \sqrt{x^3 - 1}.
    \]
    From this point of view, it appears that the degree of $F(x)$ should be $3/2$, which is how we interpret the $(2,3)$ degree of $F$.
    In general we may view the dynamics of correspondences as the iteration of generalized algebraic functions.
\end{example}

\subsection{Genus one correspondences}

In this section we analyze correspondences $F$ defined by maps between genus one curves.
The main result, Theorem \ref{thm: genus one correspondence classification}, provides simple criteria for $\C_{F^n}$ to be irreducible with genus one for all $n \geq 0$.
We use this result in our classification of exceptional correspondences in Section \ref{sec: exceptional}.

If $\O$ is an integral domain with field of fractions $K$, recall that the \textbf{inverse} of a fractional ideal $J \subseteq K$ is defined to be the fractional ideal $J\inv := \{\alpha \in K : \alpha J \subseteq \O\}$.
We say that $J$ is invertible if $J J\inv = \O$; if $J$ is not invertible, then $J J \inv$ is a proper ideal of $\O$.

\begin{lemma}
\label{lemma: invertible ideal identity}
    Let $\O$ be an integral domain.
    If $I, J \subseteq \O$ are ideals, then
    \[
        (I + J)(I \cap J) = IJ
    \]
    if and only if $I + J$ is invertible.
\end{lemma}

\begin{proof}
    We first show that the following identity holds unconditionally
    \begin{equation}
    \label{eqn: intersection product}
        I \cap J = IJ(I + J)\inv.
    \end{equation}
    Suppose $a \in I \cap J$.
    If $b \in I, c \in J$, then $ab, ac \in IJ$ and $a(b + c) \in IJ$.
    Hence $(I + J)(I \cap J) \subseteq IJ$ which implies that $I \cap J \subseteq IJ(I + J)\inv$.
    Next suppose $a \in I, b \in J, c \in (I + J)\inv$.
    Then $a, b \in I + J$, so $ac, bc \in \O$.
    Hence $abc = a(bc) = b(ac) \in I \cap J$.
    Thus $IJ(I + J)\inv \subseteq I \cap J$, completing the proof of \eqref{eqn: intersection product}.

    Multiplying both sides of \eqref{eqn: intersection product} by $I +J$ we get
    \[
        (I + J)(I \cap J) = IJ(I + J)\inv (I + J).
    \]
    The right hand side equals $IJ$ if and only if $I + J$ is invertible.
\end{proof}

\begin{lemma}
    \label{lemma: gen 1 irred criteria}
    Let $\Lambda_i \subseteq \CC$ be lattices and let $\E_i := \CC/\Lambda_i$ be elliptic curves for $i = 0, 1$.
    Let $f_1, g_1 : \E_1 \to \E_0$ be finite maps which lift to $\alpha x + \gamma$ and $\beta x + \delta$ on $\CC$.
    Then the fiber product of $f_1$ and $g_1$ is irreducible and has genus one if and only if $\alpha \Lambda_1 + \beta \Lambda_1 = \Lambda_0$.
\end{lemma}

\begin{proof}
    Let $\C$ denote the (potentially reducible) fiber product of $f_1$ and $g_1$.
    If $\Lambda_2 := \alpha \Lambda_1 \cap \beta \Lambda_1 \subseteq \Lambda_0$, then $\E_2 := \CC/\Lambda_2$ corresponds to the identity component of $\C$ and all the irreducible components are isomorphic to $\E_2$.
    We have $[\alpha\Lambda_1: \Lambda_2] = [\alpha \Lambda_1 + \beta \Lambda_1 : \beta\Lambda_1]$ by the second group isomorphism theorem.
    Then $\C$ is irreducible if and only if 
    \[
        [\alpha\Lambda_1: \Lambda_2] = [\Lambda_0 : \beta \Lambda_1] = [\Lambda_0 : \alpha \Lambda_1 + \beta\Lambda_1][\alpha \Lambda_1 + \beta \Lambda_1 : \beta\Lambda_1]
        = [\Lambda_0 : \alpha \Lambda_1 + \beta\Lambda_1][\alpha\Lambda_1 : \Lambda_2].
    \]
    Hence $\C$ is irreducible if and only if $\alpha \Lambda_1 + \beta\Lambda_1 = \Lambda_0$.
\end{proof}

\begin{lemma}
\label{lemma: invertible distributes over intersection}
    Let $\O$ be a Noetherian domain and let $J$ be a fractional $\O$-ideal.
    Then $(I_1 \cap I_2)J = I_1J \cap I_2J$ for all fractional $\O$-ideals $I_1, I_2$ if and only if $J$ is invertible.
\end{lemma}

\begin{proof}
    In general, $(I_1 \cap I_2)J = I_1J \cap I_2J$ holds for all fractional ideals $I_1, I_2$ if and only if $J$ is a flat $\O$-module.
    If $\O$ is a Noetherian domain, then $J$ flat is equivalent to $J$ invertible.
\end{proof}

The \textbf{multiplier ring} $\O_J$ of $J$ is defined to be
\(
    \O_J := \{\alpha \in K : \alpha J \subseteq J\}.
\)

\begin{lemma}
\label{lemma: proper iff invertible}
    Let $\O$ be an order in a quadratic extension $K/\QQ$ and let $\Lambda \subseteq K$ be a fractional $\O$-module.
    Then $\Lambda$ is invertible if and only if $\O$ is the multiplier ring of $\Lambda$.
\end{lemma}

\begin{proof}
    Let $\O_\Lambda$ be the multiplier ring of $\Lambda$.
    Then clearly $\O \subseteq \O_\Lambda$.
    The forward direction holds more generally:
    Suppose that $\Lambda$ is invertible and let $\alpha \in \O_\Lambda$ so that $\alpha \Lambda \subseteq \Lambda$.
    Since $\Lambda$ is invertible, we have
    \[
        \alpha \O = \alpha \Lambda \Lambda\inv \subseteq \Lambda \Lambda\inv = \O.
    \]
    Thus $\alpha \in \O$ and therefore $\O = \O_\Lambda$.

    Next suppose that $\O = \O_\Lambda$.
    Note that $K$ quadratic implies that $\Lambda \subseteq K$ is a rank 2 $\ZZ$-module.
    Replacing $\Lambda$ with a scalar multiple, we may assume that $\Lambda = \langle 1, \tau\rangle$ as a $\ZZ$-module.
    Then $\tau \in K$ is the root of an irreducible quadratic polynomial $ax^2 + bx + c \in \ZZ[x]$ where $\gcd(a,b,c) = 1$; in particular, $K = \QQ(\tau)$.
    Suppose that $\alpha := d + e\tau \in \O_\Lambda$.
    Then $\alpha = \alpha1 \in \Lambda$, which implies that $d, e \in \ZZ$, and
    \[
        \alpha \tau = d\tau + e\tau^2 = d\tau -\frac{e(c + b\tau)}{a} = -\frac{ec}{a} + \Big(d - \frac{eb}{a}\Big)\tau \in \Lambda.
    \]
    Hence $a \mid eb, ec$, which implies that $e \mid e\gcd(b,c)$.
    Since $1 = \gcd(a,b,c) = \gcd(a,\gcd(b,c))$, we conclude that $a$ divides $e$.
    Thus $\alpha = d + ae'\tau$ where $d, e' \in \ZZ$.
    Therefore $\O = \O_\Lambda = \ZZ[a\tau]$.
    Let $\bar\tau$ denote the Galois conjugate of $\tau$ and define $\overline{\Lambda} := \langle 1, \bar \tau\rangle$.
    Observe that
    \[
        \Lambda\overline\Lambda = \langle 1, \tau, \bar \tau, \tau\bar\tau\rangle = \langle 1, \frac{b}{a}, \frac{c}{a},\tau\rangle = \frac{1}{a}\langle a, b, c, a\tau\rangle.
    \]
    Since $\gcd(a,b,c) = 1$, this simplifies to
    \[
        \Lambda\overline\Lambda = \frac{1}{a}\langle 1, a\tau\rangle = \frac{1}{a}\O.
    \]
    Thus $\Lambda$ is invertible.
\end{proof}

\begin{theorem}
\label{thm: genus one correspondence classification}
    Let $\Lambda_i \subseteq \CC$ be lattices and let $\E_i := \CC/\Lambda_i$ be elliptic curves for $i = 0, 1$.
    Suppose that $F : \E_0 \rightsquigarrow \E_0$ is a correspondence such that $\E_1 := \C_F$ is irreducible with genus one.
    Suppose that $f_1, g_1 : \E_1 \to \E_0$ lift to $\alpha x + \gamma$ and $\beta x + \delta$ on $\CC$.
    \begin{enumerate}
        \item Suppose $\E_0$ does not have complex multiplication.
        \begin{enumerate}
            \item If $\E_1 \not\cong \E_0$, then $\C_{F^n}$ is reducible for all $n\geq 2$.
            
            \item If $\E_1 \cong \E_0$, then we may assume that $\Lambda_0 = \Lambda_1$ and $\alpha, \beta \in \ZZ$.
            Then $F$ is stable and $\C_{F^n}$ has genus one for all $n \geq 0$ if and only if $\gcd(\alpha, \beta) = 1$.
        \end{enumerate}
        \item Suppose $\E_i$ has complex multiplication by an order $\O_i$ in an imaginary quadratic field $K$.
        Then $\C_{F^n}$ is irreducible of genus one for all $n \geq 0$ if and only if $\alpha \Lambda_1 + \beta \Lambda_1 = \Lambda_0$ and $\alpha \O_1 + \beta \O_1$ is an invertible $\O_1$-fractional ideal.
        In particular, $F$ is stable if and only if $\C_{F^3}$ is irreducible.
    \end{enumerate}
\end{theorem}

\begin{proof}
    Suppose that $\E_0$ does not have CM.
    (1a)
    If $\E_0 \not\cong \E_1$, then $\mathrm{Hom}(\E_1,\E_0)$, the isogenies from $\E_1$ to $\E_0$, is a cyclic $\ZZ$-module generated by an isogeny $\phi$ with degree at least 2.
    Thus $f_1$ and $g_1$ have a common left factor of $\phi$, which implies that the fiber product of $f_1$ and $g_1$ is reducible.

    (1b)
    If $\E_0 \cong \E_1$, then we may assume without loss of generality that $\Lambda := \Lambda_0 = \Lambda_1$ and $\E := \E_0 = \E_1$.
    Our assumption that $\E$ does not have CM implies that $\alpha, \beta \in \ZZ$.
    Lemma \ref{lemma: gen 1 irred criteria} implies that the fiber product of $f_1, g_1$ is irreducible if and only if $\alpha \Lambda + \beta \Lambda = \Lambda$, which is equivalent to $\gcd(\alpha, \beta) = 1$.
    Consider the maps $f_2', g_2' : \E \to \E$ which lift to $\tilde f_2 := \alpha x$ and $\tilde g_2 := \beta x + \frac{\delta - \gamma}{\alpha}$.
    Then
    \[
        \tilde f_1\tilde g_2' = \alpha(\beta x + \tfrac{\delta - \gamma}{\alpha}) + \gamma
        = \alpha\beta x + \delta = \beta(\alpha x) + \delta = \tilde g_1\tilde f_2',
    \]
    which implies that $f_1 g_2' = g_1 f_2'$.
    Thus $f_2', g_2' : \E \to \E$ factors through the fiber product of $f_1, g_1$.
    Since $\alpha, \beta$ are coprime, this map to the fiber product must be an isomorphism.
    Hence, without loss of generality, $\C_{F^2} = \E$ and $f_2 = f_2'$, $g_2 = g_2'$.
    It then follows by induction that $\C_{F^n}$ is irreducible with genus one for all $n\geq 0$.

    (2)
    We prove by induction on $n$ that if $\C_{F^n} = \E_n := \CC/\Lambda_n$ and $\C_{F^{n-1}} = \E_{n-1} := \CC/\Lambda_{n-1}$ are irreducible with genus one, if $\O_i$ the multiplier ring of $\Lambda_i$ is an order in an imaginary quadratic field $K$, and if $\alpha_n, \beta_n \in K$ are elements such that $f_n, g_n : \E_n \to \E_{n-1}$ are induced by $\alpha_n \Lambda_n, \beta_n \Lambda_n \subseteq \Lambda_{n-1}$, then
    \begin{enumerate}[label=(\roman*)]
        \item $\alpha_n \Lambda_n + \beta_n \Lambda_n = \Lambda_{n-1}$, and
        \item $\alpha_n \O_n + \beta_n \O_n$ is an invertible fractional $\O_n$-ideal.
    \end{enumerate}
    The base case is precisely our hypothesis.
    Now assume this holds for some $n \geq 1$.
    Since $\E_n$ and $\E_{n-1}$ are isogenous, $\O_n$ and $\O_{n-1}$ are orders in the same imaginary quadratic field $K$.
    The assumption that $\alpha_n \Lambda_n + \beta_n \Lambda_n = \Lambda_{n-1}$ implies $\O_n \subseteq \O_{n-1}$ and that the fiber product of $f_n$ and $g_n$ is irreducible by Lemma \ref{lemma: gen 1 irred criteria}.
    Let $\Lambda_{n+1} := \alpha_n \Lambda_n \cap \beta_n\Lambda_n$ and $\E_{n+1} := \CC/\Lambda_{n+1}$.
    Then $\E_{n+1}$ is the fiber product of $f_n$ and $g_n$.
    The projections $f_{n+1}, g_{n+1} : \E_{n+1} \to \E_n$ are induced by the inclusions $\alpha_n\inv \Lambda_{n+1}, \beta_n\inv \Lambda_{n+1} \subseteq \Lambda_n$.
    The definition of $\Lambda_{n+1}$ implies that $\O_n \subseteq \O_{n+1}$.
    Thus $\Lambda_n, \Lambda_{n+1}$ are fractional $\O_n$-ideals.
    Since $\O_n$ is, by definition, the multiplier ring of $\Lambda_n$, Lemma \ref{lemma: proper iff invertible} implies that $\Lambda_n$ is an invertible fractional $\O_n$-ideal.
    Let $\alpha_{n+1} := \alpha_n\inv$ and $\beta_{n+1} := \beta_n\inv$.
    Observe that
    \begin{align*}
        \alpha_{n+1} \Lambda_{n+1} + \beta_{n+1} \Lambda_{n+1} &= (\alpha_n\inv \O_n + \beta_n\inv \O_n)(\alpha_n\Lambda_n \cap \beta_n\Lambda_n)\\
        &= (\alpha_n\inv \O_n + \beta_n\inv \O_n)(\alpha_n\O_n \cap \beta_n\O_n)\Lambda_n\\
        &= (\alpha_n\beta_n)\inv\O_n(\alpha_n\O_n + \beta_n\O_n)(\alpha_n \O_n \cap \beta_n\O_n)\Lambda_n, 
    \end{align*}
    where the second equality is a consequence of Lemma \ref{lemma: invertible distributes over intersection}.
    Our assumption that $\alpha_n \O_n + \beta_n \O_n$ is invertible implies, by Lemma \ref{lemma: invertible ideal identity}, that 
    \[
        (\alpha_n\O_n + \beta_n\O_n)(\alpha_n \O_n \cap \beta_n\O_n) = \alpha_n\beta_n \O_n.
    \]
    Hence
    \[
        \alpha_{n+1} \Lambda_{n+1} + \beta_{n+1} \Lambda_{n+1} = \Lambda_n.
    \]
    Lemma \ref{lemma: gen 1 irred criteria} implies that $\C_{F^{n+1}}$, the fiber product of $f_n$ and $g_n$, is irreducible with genus one.
    Furthermore $\O_{n+1} \subseteq \O_n$, hence $\O_n = \O_{n+1}$.
    Then
    \[
        \alpha_{n+1}\O_{n+1} + \beta_{n+1}\O_{n+1}
        = \alpha_n\inv \O_n + \beta_n\inv \O_n
        = (\alpha_n \beta_n)\inv\O_n(\alpha_n \O_n + \beta_n \O_n)
    \]
    is an invertible fractional $\O_{n+1}$-ideal.
    That completes our induction and one direction of the assertion.

    If $\alpha \Lambda_1 + \beta \Lambda_1 \neq \Lambda_0$, then $\C_{F^2}$ is reducible by Lemma \ref{lemma: gen 1 irred criteria}.
    Suppose that $\alpha \Lambda_1 + \beta \Lambda_1 = \Lambda_0$ and $\alpha \O_1 + \beta \O_1$ is not invertible.
    The Lemma \ref{lemma: invertible ideal identity} implies that
    \[
        (\alpha \O_1 + \beta \O_1)(\alpha\O_1 \cap \beta\O_1) \neq \alpha\beta \O_1.
    \]
    Hence the $n = 1$ case of our induction above implies that
    \[
        \alpha\inv \Lambda_2 + \beta\inv \Lambda_2 \neq \Lambda_1,
    \]
    which implies that $\C_{F^3}$ is reducible by Lemma \ref{lemma: gen 1 irred criteria}.
    Therefore these two conditions are both necessary as well as sufficient.
    Note that our proof shows that $\alpha \Lambda_1 + \beta \Lambda_1 = \Lambda_0$ and $\alpha \O_1 + \beta \O_1$ invertible is equivalent to $\C_{F^3}$ being irreducible. 
    Hence $F$ is stable if and only if $\C_{F^3}$ is irreducible.
\end{proof}

    

\subsection{Exceptional correspondences}
\label{sec: exceptional}

We say a correspondence $F : \PP^1 \rightsquigarrow \PP^1$ with degree $(d,e)$ is \textbf{exceptional} if $\C_{F^2}$ is irreducible with genus zero, and up to replacing $F$ by its transpose, $2 \leq d \leq e$ and $f_{1,2}$ is low-genus.
In Proposition \ref{prop: 12 implies exceptional} we show that any correspondence $F : \PP^1 \rightsquigarrow \PP^1$ for which the 12th iterate $\C_{F^{12}}$ is irreducible with genus at most one must be exceptional.
The criteria for a correspondence to be exceptional is restrictive enough that we can make an essentially complete classification of all exceptional correspondences, which we do in Theorem \ref{thm: exceptional classification}.

The next lemma shows that if $\C_F$ is irreducible with genus one for a correspondence $F : \PP^1 \rightsquigarrow \PP^1$, then the second iterate is either reducible or has large genus.

\begin{lemma}
\label{lemma: gen 0 to gen 1 transition}
    Let $F : \PP^1 \rightsquigarrow \PP^1$ be a correspondence such that $\C_F$ is irreducible with genus one.
    Then either $\gen(\C_{F^2}) > 1$ or $\C_{F^2}$ is reducible.
\end{lemma}

\begin{proof}
    Suppose $\C_{F^2}$ is irreducible.
    Riemann-Hurwitz implies that $\gen(\C_{F^2}) \geq 1$.
    Suppose that $\gen(\C_{F^2}) = 1$.
    Since $\C_F$ and $\C_{F^2}$ both have genus one, Riemann-Hurwitz implies that $f_2, g_2$ are both unramified.
    Then Abhyankar's lemma implies that $f_1$ and $g_1$ have essentially the same special ramification type:
    If $p$ is a critical value of $f_1$, then $m_{f_1}(p) \mid e_{g_1}(q)$ for all $q \in g_1\inv(p)$.
    By symmetry we get the same conclusion after swapping the roles of $f_1$ and $g_1$.
    Hence $f_1$ and $g_1$ have the same set of critical values and $m_{f_1}(p) = m_{g_1}(p)$ for all $p$.
    The fact that $f_1$ is evenly ramified over each of its critical values implies the same is true for all of its Galois conjugates; the Galois closure of $f_1$ is an irreducible component of the fiber product of these conjugates.
    Hence Abhyankar's lemma implies that the Galois closure of $f_1$ is unramified over $\C_F$.
    Thus the Galois closure of $f_1$ has genus one, and the same holds for $g_1$.
    Lemma \ref{lemma: canonical factorization} implies that $f_1$ has a left factor $\pi$, which is Galois and shares the same critical values and signature as $f_1$.
    Note that the critical values and signature of $\pi$ determine $\pi$ uniquely up to right composition with an isomorphism.
    Thus $\pi$ must be a common left factor of both $f_1$ and $g_1$.
    However, in that case, the fiber product of $f_1$ and $g_1$ is reducible.
    This contradiction implies that either $\gen(\C_{F^2}) > 1$ or $\C_{F^2}$ is reducible.
\end{proof}

\begin{proposition}
\label{prop: 12 implies exceptional}
    Let $F : \PP^1 \rightsquigarrow \PP^1$ be a correspondence with degree $(d,e)$ where $d, e \geq 2$.
    If $\C_{F^{12}}$ is irreducible and $\gen(\C_{F^{12}}) \leq 1$, then $F$ is exceptional.
\end{proposition}

\begin{proof}
    Suppose that $\C_{F^{12}}$ is irreducible with genus at most one.
    Up to taking a transpose, we may assume that $2 \leq d \leq e$.
    Note that $\C_{F^{12}}$ is the fiber product of $f_{1,2}$ and $g_{1,10}$.
    A calculation shows that $d^{10} - 170d^2 + 84 > 0$ for all $d \geq 2$
    Hence we have
    \[
        \deg g_{1,10} = e^{10} \geq d^{10}> 170d^2 - 84 = 170 \deg f_{1,2} - 84.
    \]
    Hence Corollary \ref{corollary: low genus criteria} implies that $f_{1,2}$ is low-genus.
    Since $\C_{F^{12}}$ is irreducible with genus at most one, the same is true for all $\C_{F^n}$ with $n \leq 12$.
    Thus Lemma \ref{lemma: gen 0 to gen 1 transition} implies that $\C_{F^2}$ must have genus 0, and we conclude that $F$ is exceptional.
\end{proof}

We are now in a position to prove our main result on the scarcity of $K$-rational points in orbits under correspondences for finitely generated fields $K$.

\begin{theorem}
\label{thm: main}
    Let $K$ be a finitely generated field of characteristic zero.
    Let $F : \PP^1 \rightsquigarrow \PP^1$ be a correspondence defined over $K$ such that $\C_{F^{12}}$ is irreducible.
    Then either $F$ is exceptional, or for all $n \geq 12$, there are finitely many $K$-rational points $p$ for which $F^{n}(p)$ contains a $K$-rational point.
\end{theorem}

\begin{proof}
    If $p \in \PP^1(K)$ and $q \in F^n(p) \cap \PP^1(K)$, then there is at least one $K$-rational point $r \in \C_{F^n}(K)$ such that $p = f_{1,n}(r)$ and $q = g_{1,n}(r)$. 
    Hence if there are infinitely many such $K$-rational points $p$, then $\C_{F^n}(K)$ is infinite.
    Thus Faltings's theorem \cite[Thm. 3]{faltings1992complements} implies that $\C_{F^n}$ has a component with genus at most one.
    Assuming $n \geq 12$, Lemma \ref{lemma: RH cor} implies that $\C_{F^{12}}$ must also have genus at most one.
    Thus $F$ is exceptional by Proposition \ref{prop: 12 implies exceptional}.
\end{proof}

Our final result classifies exceptional correspondences up to isomorphism.
There are several cases where $F$ is exceptional but $\C_{F^n}$ is reducible for some small $n$ which we defer from analyzing in detail due to the tediousness of the case analysis involved.
Notably this happens when $\deg f_1 = 2$ or when $f_1$ is a degenerate Latt\'es map (a generalized Latt\'es map with a genus zero Galois closure).

\begin{theorem}
\label{thm: exceptional classification}
    Let $F$ be an exceptional correspondence and suppose that $d := \deg f_1$.
    \begin{enumerate}
        
        \item If $f_{1,2}$ has cyclic signature, then $F$ is isomorphic to $f_1 = x^d$ and $g_1 = x^e h(x)^d$ for some $e$ coprime to $d$ and some rational function $h$ with $x$-valuation 0.
        \begin{enumerate}
            \item $F$ is stable.
            \item If $h$ is constant, then $g_1 = c x^e$ and $\gen(\C_{F^n}) = 0$ for all $n \geq 0$.
            \item If $h$ is non-constant and $m$ is the largest integer such that $g_1 = x^e \tilde h(x)^{d^{m-1}}$ for some polynomial $\tilde h(x)$, then $\gen(\C_{F^n}) = 0$ for all $n \leq m$ and if $d \geq 3$ or $d =2$ and $m \geq 3$, then $\gen(\C_{F^n}) > 1$ for $n > m$. 
            If $d = m = 2$, then $\gen(\C_{F^n}) > 1$ for all $n > m + 1$.
        \end{enumerate}
        
        \item If $d \geq 3$ and $f_{1,2}$ has dihedral signature, then up to isomorphism of $F$, $f_1 = T_d$ and $g_1$ is a rational function for which $g_1 \pi = \pi \tilde g_1$, where $\pi = \frac{x + x\inv}{2}$, $\tilde g_1 = \pm x^e \frac{s(x)^d}{s^\dagger (x)^d}$, the $x$-valuation of $s(x)$ is 0, and $s^\dagger(x) := x^{\deg s}s(x\inv)$.
        \begin{enumerate}
            \item $F$ is stable.
            \item If $s$ is constant, then $\gen(\C_{F^n}) = 0$ for all $n \geq 0$.
            \item If $s$ is non-constant and $m$ is the largest integer such that $g_1 = x^e \frac{\tilde s(x)^{d^{m-1}}}{\tilde s^\dagger(x)^{d^{m-1}}}$ for some polynomial $\tilde s(x)$, then $\gen(\C_{F^n}) = 0$ for all $n \leq m$ and $\gen(\C_{F^n}) > 1$ for $n > m$. 
        \end{enumerate}
        
        \item If $d \geq 3$ and $f_{1,2}$ has a Latt\'es signature, then either $\C_{F^{11}}$ is reducible, or has genus $> 1$, or $(f_1,g_1) = (L_{\pi_1,\pi_2,\psi}, L_{\pi_1,\pi_2,\phi})$ for some $\pi_i : \E_i \to \PP^1$ and some finite maps $\psi, \phi : \E_1 \to \E_0$.
        In the latter case, if $\C_{F^3}$ is irreducible, then $F$ is stable and $\C_{F^n}$ has genus zero for all $n \geq 0$.

        \item If $d = 2$, then either $\C_{F^{16}}$ is reducible, or has genus $> 1$, or $f_{1,4}$ is a low-genus rational function.
        In the latter case, $f_{1,2}$ has degree 4 and one of the previous cases applies to $F^2$.
        
        \item $f_{1,2}$ cannot have a Platonic signature.
    \end{enumerate}
\end{theorem}

\begin{proof}

    (1) 
    If $f_{1,2}$ has cyclic signature, then Lemma \ref{lemma: Galois duck} implies that $f_{1,2}$ is Galois with cyclic Galois group.
    Up to conjugation, there is only one faithful action of the cyclic group on $\PP^1$.
    Hence we may choose M\"obius transformations $\mu_1, \mu_3$ such that $\mu_1 f_{1,2} \mu_3\inv = x^{d^2}$.
    Let $p_q := f_2\mu_3\inv(q)$ for $q \in \{0,1,\infty\}$.
    Then $\mu_1 f_1f_2 \mu_3\inv = x^{d^2}$ implies that $\mu_1 f_1(p_q) = q$.
    Let $\mu_2$ be the unique M\"obius transformation such that $\mu_2(p_q) = q$ for $v \in \{0,1,\infty\}$.
    Then $\mu_1 f_1 \mu_2\inv$ and $\mu_2 f_2 \mu_3\inv$ have the same critical values, critical points, and ramification as $x^{d}$.
    This, together with $\mu_1 f_1 \mu_2\inv(1) = \mu_2 f_2 \mu_3\inv(1) = 1$ implies that $\mu_1 f_1 \mu_2\inv = \mu_2 f_2 \mu_3\inv = x^{d}$.
    Hence up to isomorphism of $F$ we may assume that $f_1 = f_2 = x^d$.
    
    Abhyankar's lemma implies that
    \begin{enumerate}[label=(\roman*)]
        \item $g_1(\{0,\infty\}) = \{0,\infty\}$,
        \item $e_{g_1}(0), e_{g_1}(\infty)$ are coprime to $d$, and
        \item $d \mid e_{g_1}(p)$ for all $p \in g_1\inv(\{0,\infty\}) \setminus \{0,\infty\}$.
    \end{enumerate}
    These conditions are equivalent to $g_1 = x^eh(x)^d$ for some integer $e$ coprime $d$ and some rational function $h$.
    In particular, this implies that $\deg g_1$ is coprime to $d$, hence that $F$ is stable by Lemma \ref{lemma: coprime degree implies irreducible}, proving (1a).

    (1b)
    If $h = c$ is constant, then $g_1 = cx^e$ and $g_2 = c^{1/d} x^e$.
    Then it follows by induction that $f_n = x^d$ and $g_n = c^{1/d^{n-1}}x^e$ for all $n\geq 1$.
    Thus $\gen(\C_{F^n}) = 0$ for all $n \geq 0$.

    (1c)
    If $h$ is non-constant, then there is some integer $m \geq 2$ such that $h = \tilde h(x)^{d^{m-1}}$ for a polynomial $\tilde h(x)$ which is not a $d$th power.
    The only way that $\tilde h(x^d)$ can be a $d$th power is if $\tilde h(x) = x^{e'}h'(x)^d$ for some integer $e'$ not divisible by $d$ and some rational function $h'$.
    However, we assumed that the $x$-valuation of $h$, and hence of $\tilde h$, was 0.
    Thus $h(x^d)$ is not a $d$th power.
    
    Then $g_2 = x^e \tilde h(x^d)^{d^{m-2}}$ and it follows by induction that $f_n = x^d$ and $g_n = x^e \tilde h(x^{d^{n-1}})^{d^{m-n}}$ for $n \leq m$.
    Our assumption about $\tilde h$ implies that there is at least one point $p \in \tilde h\inv(\{0,\infty\}) \setminus \{0,\infty\}$ for which $e_{\tilde h}(p)$ is not divisible by $d$.
    Hence if $d \geq 3$ or if $d = 2$ and $m \geq 3$, then there are at least $d^{m-1} \geq 3$ points $q$ for which the ramification of $\tilde h(x^{d^{m-1}})$ at $q$ is not divisible by $d$.
    Abhyankar's lemma implies that $f_{m+1}$ has at least 5 critical points, two of which are totally ramified.
    Then Riemann-Hurwitz implies that
    \[
        2\gen(\C_{F^{m+1}}) + 2d - 2 = \sum_{p}e_{f_{m+1}}(p) - 1 \geq 2d +1 \Longrightarrow \gen(\C_{F^{m+1}}) > 1.
    \]
    Therefore $\gen(\C_{F^n}) > 1$ for all $n > m$.
    If $d = m = 2$, then $\C_{F^{m + 1}}$ could have genus one, but then Lemma \ref{lemma: gen 0 to gen 1 transition} implies that $\gen(\C_{F^n}) > 1$ for all $n > m + 1$.

    (2) 
    Suppose $f_{1,2}$ has a dihedral signature.
    Since $d \geq 3$, $f_1$ must also have dihedral signature.
    If $f_1$ is Galois, then so is $f_2$.
    However, the composition of Galois dihedral maps with $d \geq 3$ does not have dihedral signature.
    Hence $f_1$ is not Galois.
    Thus, up to isomorphism of $F$, we may assume that $f_1 = T_d$.

    If $f_2$ is Galois, then Abhyankar's lemma implies that all the ramification indices of $g_1$ over $\pm 1$ are even.
    Thus Lemma \ref{lemma: power map fiber product} implies that $g_1 = \pi h_1$ where $\pi = \frac{x + x\inv}{2}$ and $h_1$ is a rational function.
    Thus we have the following fiber product diagram.
    \[
        \begin{tikzcd}
        \mathbb{P}^1 \arrow[d, "T_d"'] & \mathbb{P}^1 \arrow[l, "\pi"'] \arrow[d, "x^d"] & \mathbb{P}^1 \arrow[l, "h_2"'] \arrow[d, "f_2"] \\
        \mathbb{P}^1                   & \mathbb{P}^1 \arrow[l, "\pi"]                   & \mathbb{P}^1 \arrow[l, "h_1"]                  
        \end{tikzcd}
    \]
    Then $f_2 = \mu_1 x^d \mu_2\inv$ for some M\"obius transformations $\mu_1, \mu_2$.
    However, $d \geq 3$ implies that $f_{1,2} = T_d \mu_1 x^d \mu_2\inv$ does not have dihedral signature.
    Thus $f_2$ is not Galois.

    Therefore $f_2$ must have the same ramification type as $T_d$ and $f_{1,2}$ must have the same ramification type as $T_{d^2}$.
    Since $f_2$ is only determined up to right composition with a M\"obius transformation, we may assume that $f_{1,2} = T_{d^2}$.
    Then $T_df_2 = T_dT_d$.
    If $d$ is odd, this immediately implies that $f_2 = T_d$.
    If $d$ is even, then we conclude that $f_2 = \pm T_d$.
    However, in that case we may replace $(f_1, f_2)$ by $(f_1\mu\inv, \mu f_2)$ where $\mu = -x$ in order to assume $f_1 = f_2 = T_d$.
    
    Abhyankar's lemma implies that for $i = 1, 2$
    \begin{enumerate}[label=(\roman*)]
        \item $g_i(\{\pm 1\}) = \{\pm 1\}$,
        \item $g_i(\infty) = \infty$,
        \item $e_{g_i}(\pm1)$ are both odd,
        \item $e_{g_i}(\infty)$ is coprime to $d$,
        \item $e_{g_i}(p)$ is even for all $p \in g_i\inv(\{\pm 1\}) \setminus \{\pm 1\}$, and
        \item $d \mid e_{g_i}(p)$ for all $p \in g_i\inv(\infty)\setminus\{\infty\}$.
    \end{enumerate}
    Recall that $T_d\pi = \pi(x^d)$.
    Since $\pi$ is degree 2 and totally ramified over $\pm 1$, Lemma \ref{lemma: power map fiber product} implies that the fiber product of $\pi$ with $g_1$ is irreducible and Abhyankar's lemma implies that there is some rational function $\tilde g_i$ such that $g_i\pi = \pi \tilde g_i$.
    Hence we have the following commutative diagram in which the vertical maps are all $\pi$ and each square is a fiber product.
    \[
        \begin{tikzcd}[row sep=2em, column sep=2em]
        & \PP^1 \arrow[dl, "x^d"'] \arrow[dd] 
            & & \PP^1 \arrow[ll, "\tilde{g}_2"'] \arrow[dl, "x^d"] \arrow[dd] \\
        \PP^1 \arrow[dd] 
            & & \PP^1 \arrow[ll, "\tilde{g}_1"' near start, crossing over]\\
        & \PP^1  \arrow[dl, "T_d"'] 
            & & \PP^1 \arrow[dl, "T_d"] \arrow[ll, "g_2"' near end] \\
        \PP^1  
            & & \PP^1 \arrow[ll, "g_1"]
        \arrow[from=2-3, to=4-3, crossing over]
        \end{tikzcd}
    \]
    As we argued in part (2), $\tilde g_1 = x^eh(x)^d$ for some integer $e$ coprime to $d$ and some rational function $h$ with $x$-valuation 0.
    This diagram implies that $\pi \tilde g_1(x) = \pi \tilde g_1(x\inv)$, hence that $\tilde g_1(x\inv) = \tilde g_1(x)^{\pm 1}$.
    If $\tilde g_1(x) = \tilde g_1(x\inv)$, then $\tilde g_1 = \tilde h_1 \pi$ for some rational function $\tilde h_1$.
    However, in this case $g_1 = \pi \tilde h_1$, which contradicts the fact that $g_1$ and $\pi$ have an irreducible fiber product.
    Therefore $\tilde g_1(x\inv) = \tilde g_1(x)\inv$, or equivalently that $\zeta_dh(x)h(x\inv) = 1$ for some $d$th root of unity $\zeta_d$.
    Let $h = \frac{s}{t}$ where $s, t$ are coprime polynomials in $x$ which do not vanish at $x = 0$.
    Then $\zeta_ds(x)s(x\inv) = t(x)t(x\inv)$.
    This implies that $\deg s = \deg t$.
    Let $s^\dagger(x) := x^{\deg s}s(x\inv)$ denote the reciprocal polynomial of $s$.
    Then $\zeta_ds(x)s^\dagger(x) = t(x)t^\dagger(x)$ and $s$ coprime to $t$ implies that $t(x) = \sqrt\zeta_d s^\dagger(x)$.
    Therefore $\tilde g_1 = \pm x^e \frac{s(x)^d}{s^\dagger(x)^d}$ for $s(x)$ a polynomial with $x$-valuation 0.

    Hence $\deg g_1 = \deg \tilde g_1$ is coprime to $d$ and Lemma \ref{lemma: coprime degree implies irreducible} implies that $F$ is stable.
    If $s$ is constant, then (1b) implies (2b), and (1c) implies that $\C_{F^n}$ has genus 0 for $n \leq m$.
    Suppose that $s$ is nonconstant.
    Then there exists a point $p \in g_m\inv(\infty)\setminus\{\infty\}$ such that $d \mid e_{g_m}(p)$ and $d^2 \nmid e_{g_m}(p)$.
    Therefore Abhyankar's lemma implies that $g_{m+1}\inv(\infty)$ contains at least $d \geq 3$ points $q$ such that $e_{g_{m+1}}(q)$ is not divisible by $d$.
    Note that $g_{m+1}$ has identical ramification to $g_m$ over $\pm 1$ and $\infty$.
    Let $k := \deg g_1$, then Riemann-Hurwitz applied to $g_{m+1}$ gives us
    \[
        2\gen(\C_{F^{m+1}}) + 2k - 2 = \sum_{p}e_{g^{m+1}}(p) - 1 \geq 2k +1 \Longrightarrow \gen(\C_{F^{m+1}}) > 1.
    \]
    Thus $\gen(\C_{F^n}) > 1$ for all $n > m$.

    (3) 
    Suppose that $f_{1,2}$ has a Latt\'es signature.
    Then $f_{1,2}$ has Galois closure of genus one.
    \footnote{While it is true by Lemma \ref{lemma: normalizing lattes} that $f_1$ and $f_2$ also generalized Latt\'es maps, they may be degenerate in the sense that their Galois closures have genus zero.
    This leads to a handful of special cases which after a few more iterates either resolve to the general case or lead to a reducible or genus $> 1$ iterate.
    What we argue here is that if $f_{1,3}$ is a generalized Latt\'es map, then that eliminates all the special cases.}
    Suppose $\C_{F^{11}}$ is irreducible with genus at most one.
    Note that $\C_{F^{11}}$ is the fiber product of $f_{1,3}$ and $g_{1,8}$.
    A calculation shows that $d^8 > 170d^3 - 84$ for all $d \geq 3$.
    Then
    \[
        \deg g_{1,8} = e^8 \geq d^8 > 170 d^3 - 84 = 170\deg f_{1,3} - 84,
    \]
    together Corollary \ref{corollary: low genus criteria} and Lemma \ref{lemma: gen 0 to gen 1 transition} implies that $f_{1,3}$ is a low-genus rational function.
    Since $f_{1,2}$ is a left composition factor of $f_{1,3}$, the Galois closure of $f_{1,3}$ also has genus one.
    Lemma \ref{lemma: normalizing lattes} implies that $f_{1,3} = L_{\pi_1, \pi_4,\psi}$ for some $\pi_i : \E_i \to \PP^1$ and some $\psi : \E_4 \to \E_1$.
    Let $\E_2$ and $\E_3$ be the irreducible components of the fiber products of $\pi_1$ with $f_1$ and $f_{1,2}$, respectively, which $\psi$ factors through.
    Let $\pi_2 : \E_2 \to \PP^1$ and $\pi_3 : \E_3 \to \PP^1$ be the restrictions of the projections from the fiber product.
    Since $\pi_1$ is a cyclic Galois map, the second group isomorphism theorem implies that $\pi_i$ for $2\leq i \leq 4$ are as well.
    Let $\psi = \psi_1\psi_2\psi_3$ be the factorization of $\psi$ through $\E_2$ and $\E_3$.
    Thus we have the following commutative diagram,
    \[
        \begin{tikzcd}
        \mathcal{E}_1 \arrow[d, "\pi_1"'] & \mathcal{E}_2 \arrow[l, "\psi_1"'] \arrow[d, "\pi_2"] & \mathcal{E}_3 \arrow[l, "\psi_{2}"'] \arrow[d, "\pi_4"] & \E_4 \arrow[l, "\psi_3"'] \arrow[d,"\pi_4"]\\
        \mathbb{P}^1                      & \mathbb{P}^1 \arrow[l, "f_1"]                         & \mathbb{P}^1 \arrow[l, "f_{2}"] & \PP^1 \arrow[l,"f_3"]         
        \end{tikzcd}
    \]
    which shows that $f_i = L_{\pi_i, \pi_{i+1}, \psi_i}$ for $1 \leq i \leq 3$.
    We now turn our attention to this fiber product diagram.
    \[
        \begin{tikzcd}
        \mathbb{P}^1 \arrow[d, "{f_{1,2}}"'] & \mathbb{P}^1 \arrow[l, "g_3"'] \arrow[d, "{f_{2,3}}"] \\
        \mathbb{P}^1                         & \mathbb{P}^1 \arrow[l, "g_1"]                        
        \end{tikzcd}
    \]

    First suppose that $f_{1,2}$ and $f_{2,3}$ have different signatures.
    This can only happen if $f_{1,2}$ has signature $(2,4,4)$ or $(2,3,6)$, and $f_{2,3}$ has signature $(2,2,2,2)$ in the first case and either $(2,2,2,2)$ or $(3,3,3)$ in the second case.
    Both cases require there to be a prime $p = 2$ or 3 and two points $v_i, v_j \in V$ such that $p \mid e_{g_1}(q)$ for every point $q \in g_1\inv(\{v_i,v_j\})$ and $p \mid e_{f_{1,2}}(q)$ for every point $q \in f_{1,2}\inv(\{v_i,v_j\})$.
    But, in that case, Lemma \ref{lemma: power map fiber product} implies that $f_{1,2}$ and $g_1$ have a common left factor of the form $\mu x^p \mu\inv$ for some M\"obius transformation $\mu$.
    Thus the fiber product of $f_{1,2}$ and $g_1$ is reducible---a contradiction.

    Thus $f_{1,2}$ and $f_{2,3}$ have the same signature.
    Let $e := \deg g_1$, let $V = \{v_1,\ldots, v_k\}$ denote the set of critical values of $f_{1,2}$, and let $m_i := m_{f_{1,2}}(p_i)$.
    Riemann-Hurwitz applied to $g_1$ tells us that
    \[
        2e - 2 = ke - |g_1\inv(V)| + R,
    \]
    where $R$ is the contribution coming from ramification of $g_1$ over points not in $V$.
    Solving for $R$ we have
    \[
        R = |g_1\inv(V)| - ke + 2(e - 1).
    \]
    Abhyankar's lemma and our assumption that $f_{2,3}$ has the same signature as $f_{1,2}$ implies that $g_1\inv(V)$ contains $k$ points $p$ with $e_{g_1}(p)$ coprime to $m_i$ for $p \in g_1\inv(v_i)$.
    All the other points $q \in g_1\inv(V)$ must have $m_i \mid e_{g_1}(q)$ when $q = g_1\inv(v_i)$.
    Thus
    \[
        |g_1\inv(V)| \leq k + (e -1)\sum_{i=1}^k\frac{1}{m_i}.
    \]
    Combined with our expression for $R$ this yields
    \[
        R \leq (e-1)\Big(2 - k + \sum_{i=1}^k\frac{1}{m_i}\Big).
    \]
    If $k = 4$, then each $m_i = 2$ and if $k = 3$, then $\sum_{i=1}^3 \frac{1}{m_i} = 1$, hence in either case $R = 0$.
    Therefore $g_1$ is only ramified over $V$, $m_{g_1}(v_i) = m_i$ for each $i$, and $\sum_{i=1}^k\frac{1}{m_i} = k - 2$.
    These ramification constraints imply that $g_1 = L_{\pi_1, \pi_2, \psi'}$ for some $\psi' : \E_2 \to \E_1$.
    Let $\widetilde F$ be the correspondence defined by $\psi_1, \psi' : \E_2 \to \E_1$.
    If $\C_{F^3}$ is irreducible, then so is $\C_{\widetilde F^3}$.
    Then Theorem \ref{thm: genus one correspondence classification} implies that $\widetilde F$ is stable and $\C_{\widetilde F^n}$ has genus one for all $n \geq 0$.
    Hence $F$ is also stable and $\C_{F^n}$ has genus zero for all $n \geq 0$.

    (4)
    Suppose that $d = 2$.
    Suppose $\C_{F^{16}}$ is irreducible with genus at most one.
    Note that $\C_{F^{16}}$ is the fiber product of $f_{1,4}$ and $g_{1,12}$.
    A calculation shows that $d^{12} > 170d^4 - 84$ for $d = 2$.
    Then
    \[
        \deg g_{1,12} = e^{12} \geq d^{12} > 170 d^4 - 84 = 170\deg f_{1,4} - 84,
    \]
    together Corollary \ref{corollary: low genus criteria} and Lemma \ref{lemma: gen 0 to gen 1 transition} implies that $f_{1,4}$ is a low-genus rational function.

    (5) 
    Suppose that $f_{1,2}$ has a Platonic signature.
    Let $G$ be the Galois group of the Galois closure of $f_{1,2}$.
    Then by assumption $G$ is isomorphic to either $A_4, S_4$, or $A_5$.
    Galois theory implies that there are subgroups $H_0 \subseteq H_1 \subseteq G$ such that $[G : H_1] = [H_1 : H_0] = d$.
    Thus $d^2$ divides $|G|$ for some $d \geq 2$, which for these groups implies that $d = 2$.
    However $A_4$ and $A_5$ do not have any index 2 subgroups, and the only index 2 subgroup of $S_4$ is $A_4$.
    Thus this case is impossible.
\end{proof}

\printbibliography

\end{document}